\documentclass[twocolumn]{autart}

\usepackage{graphicx}
\usepackage{float}
\usepackage{booktabs}
\usepackage{tabularx}
\usepackage{array}
\usepackage{adjustbox}
\usepackage{multirow}
\usepackage{placeins}
\usepackage{tikz}
\usetikzlibrary{arrows.meta,positioning,shapes.geometric}
\usepackage{subcaption}
\usepackage{algorithm}
\usepackage{algpseudocode}
\usepackage{amssymb}
\let\theoremstyle\relax          
\usepackage{amsthm}          
\usepackage[font=small,labelfont=bf]{caption}

\newcolumntype{L}[1]{>{\raggedright\arraybackslash}p{#1}}
\newcolumntype{C}[1]{>{\centering\arraybackslash}p{#1}}

\newcommand{\pplus}{\mathbin{+}\allowbreak}
\newcommand{\mminus}{\mathbin{-}\allowbreak}

\theoremstyle{plain}
\newtheorem{proposition}{Proposition}[section]
\newtheorem{lemma}{Lemma}[section]

\theoremstyle{remark}
\newtheorem{remark}{Remark}[section]

\renewcommand{\arraystretch}{1.05}
\usepackage{mathtools} 

\begin{document}

\begin{frontmatter}

\title{Multiparametric continuous-time optimal control via Pontryagin's Maximum Principle: 
explicit solutions and comparisons with discrete-time formulations\thanksref{footnoteinfo}} 

\thanks[footnoteinfo]{This paper was not presented at any IFAC 
meeting. Corresponding author  E. N. Pistikopoulos E-mail. stratos@tamu.edu. 
}

\author[TAMUE,TAMUC]{Lida Lamakani}\ead{lida.lamakani@tamu.edu},    
\author[TAMUE,TAMUC]{Efstratios N. Pistikopoulos}\ead{stratos@tamu.edu}             

\address[TAMUE]{Texas A\&M Energy Institute, Texas A\&M University, College Station, USA}  
\address[TAMUC]{Artie McFerrin Department of Chemical Engineering, Texas A\&M University, College Station, USA}             

\begin{abstract}
Model predictive control offers a powerful framework for managing constrained systems, but its repeated online optimization can become computationally prohibitive. Multiparametric programming addresses this challenge by precomputing optimal solutions offline, enabling real-time control through simple function evaluation. While extensively developed for discrete-time systems, this approach suffers from combinatorial growth in solution complexity as discretization is refined. This paper presents a systematic continuous-time multiparametric framework for linear-quadratic optimal control that directly solves Pontryagin's optimality conditions without discretization artifacts. Through two illustrative examples, we demonstrate that continuous-time formulations yield solutions with substantially fewer critical regions than their discrete-time counterparts. Beyond this reduction in partition complexity, the continuous-time approach provides deeper insight into system dynamics by explicitly identifying switching times and eliminating discretization artifacts that obscure the true structure of optimal control policies. Knowledge of the switching structure also accelerates online optimization methods by providing analytical information about the solution topology. Clear step-by-step algorithms are provided for identifying switching structures, computing parametric switching times, and constructing critical regions, making the continuous-time framework accessible for practical implementation.
\end{abstract}

\end{frontmatter}

\section{Introduction}

Model predictive control (MPC) is widely applied to complex constrained systems. The approach works by using a model to predict how the system will behave, then choosing control actions that optimize performance while staying within operational limits \cite{rawlings2020model}. This has made MPC especially valuable in settings where safety and physical constraints matter—chemical plants, automotive systems, robotics, and finance \cite{muske1993model,richalet1993industrial,falcone2007predictive,diehl2006fast,dombrovskii2018model}.

MPC is typically implemented through a receding horizon strategy \cite{rawlings2020model}. At each step, only the first control action is applied before 
re-solving with the updated state, requiring a full 
optimization at every sampling instant.

This computational burden naturally leads to the following question: can the optimization be solved once offline, rather than repeatedly at each time step? Multiparametric programming offers a solution by treating the current state as a parameter, the optimization problem can be solved for all feasible states offline. Pistikopoulos, Bemporad, and colleagues \cite{pistikopoulos2002line,bemporad2002explicit} showed that the optimal control law can be expressed as a piecewise affine function over a partition of the state space. Online control then reduces to region identification and function evaluation, eliminating real-time optimization. This explicit MPC framework has been extended to hybrid systems \cite{borrelli2005explicit,oberdieck2015explicit}, uncertain and robust systems \cite{bemporad2003min,sakizlis2004design,kouramas2013algorithm}, and nonlinear systems \cite{pistikopoulos2020multi,charitopoulos2017closed,charitopoulos2017nonlinear}, demonstrating substantial computational benefits.

The majority of explicit MPC research has considered discrete-time formulations. These lead to finite-dimensional optimization problems that are computationally tractable, explaining their popularity. A key limitation, however, is that most physical systems are inherently continuous-time. Chemical reactors, vehicle dynamics, robotic manipulators, power systems, and biological processes are all naturally modeled by differential equations. Discretization introduces approximations, and while finer time grids improve accuracy, they also increase the problem dimension. Critically, the number of critical regions in the parametric partition tends to grow substantially with the number of discretization points, increasing both offline computation and online lookup cost. Can multiparametric methods be developed directly for continuous-time systems to avoid this combinatorial explosion?

A continuous-time formulation offers an alternative approach that preserves the natural structure of the underlying dynamics. Rather than discretizing, the optimization is formulated directly in terms of the differential equations. Pontryagin's Maximum Principle \cite{pontryagin1962mathematical} provides necessary conditions for optimality in the form of a two-point boundary value problem. For linear-quadratic systems, these conditions admit analytical solutions. In the presence of constraints, the optimal solution exhibits a piecewise structure: different analytical expressions govern the solution over distinct time intervals, with transitions occurring when constraints become active or inactive. The switching times depend on the initial state, which is precisely the parametric dependence required for explicit MPC.

When path constraints are active, the optimality conditions form a differential-algebraic equation (DAE) system combining the dynamics with algebraic constraint equations \cite{brenan1995numerical}. 
Biegler and colleagues have developed simultaneous approaches that discretize DAE systems and solve large-scale nonlinear programs \cite{biegler2012control}; sensitivity-based strategies have been proposed to reduce the associated online computational cost \cite{zavala2009advanced,zavala2008fast,de1995extension}, though the discretization-dependent complexity discussed above remains. Bonvin and coworkers have developed complementary methods for batch processes that systematically exploit the structure of optimality conditions to handle path constraints \cite{srinivasan2003dynamic,srinivasan2003,franccois2005use}. Their work demonstrates that understanding which constraints are active and when is crucial for both computational efficiency and physical insight—a perspective that aligns directly with the multiparametric continuous-time framework developed here.

Sakizlis and coworkers \cite{sakizlis2004design,sakizlis2005explicit} pioneered this approach for linear continuous-time systems, demonstrating that the parameter space can be partitioned into critical regions where the switching times vary continuously as functions of the parameters while the sequence of active constraints remains fixed. Within each region, the optimal trajectories have explicit analytical forms—typically exponentials in time multiplied by functions of parameters—in contrast to the piecewise affine laws obtained in discrete-time formulations.

Building on this foundation, the present work makes three contributions. First, we provide complete step-by-step algorithms that transform the continuous-time approach from a theoretical concept into a practical computational method. We reformulate the jump conditions using state continuity instead of Hamiltonian continuity, which is simpler to implement and more numerically stable. Second, we conduct the first systematic study across multiple discretization levels, showing that continuous-time formulations consistently require far fewer critical regions—often by an order of magnitude—while discrete-time formulations exhibit combinatorial growth as discretization is refined. Third, we demonstrate that continuous-time solutions provide explicit analytical expressions for switching times and optimal trajectories, offering direct physical insight and enabling more efficient online optimization.

In both examples, the continuous-time partition holds at 5 critical regions regardless of accuracy requirements, while the discrete-time count grows from 11 to over 60 as discretization is refined. This difference has direct practical consequences: online region lookup through 5 regions is substantially faster than through 60, which matters for systems operating under tight computational budgets.

The paper is organized as follows. Section~2 formulates the continuous-time optimal control problem and reviews optimality conditions. Section~3 develops the multiparametric framework. Section~4 presents two illustrative examples comparing continuous-time and discrete-time approaches. Section~5 discusses implementation and computational aspects. Section~6 concludes.
\section{Continuous-Time Optimal Control Structure}
This section formulates both continuous-time and discrete-time optimal control problems and reviews the corresponding optimality conditions.
\subsection{Continuous-Time Optimal Control Problem}
Following the standard formulation in \cite{sakizlis2005explicit,Pistikopoulos2007}, the system dynamics are described by
\begin{equation}
\dot{x}(t) = f\big(x(t),u(t)\big),
\label{eq:ct_dynamics}
\end{equation}
where $x(t)\in\mathbb{R}^n$ denotes the state vector and $u(t)\in\mathbb{R}^{n_u}$ is the control input. For a given initial condition $x(t_0)=x$, the continuous-time optimal control problem is posed as
\begin{align}
\min_{x(\cdot),\,u(\cdot)} \quad
& J_c(x) = \int_{t_0}^{t_f} \ell\big(x(t),u(t)\big)\,dt
+ \phi\big(x(t_f)\big)
\label{eq:ct_cost} \\[0.4em]
\text{s.t.}\quad
& \dot{x}(t) = f\big(x(t),u(t)\big),
\label{eq:ct_ocp_dynamics} \\
& g\big(x(t),u(t)\big) \le 0,
\quad \forall\, t\in[t_0,t_f],
\label{eq:ct_path_constraints} \\
& \psi\big(x(t_f)\big) \le 0,
\label{eq:ct_terminal_constraints} \\
& x(t_0) = x .
\label{eq:ct_initial_condition}
\end{align}
For comparison, we also consider a discrete-time formulation obtained by sampling the dynamics with period $\Delta t$, following standard discrete-time optimal control and MPC formulations \cite{pistikopoulos2020multi}. The sampled model is given by
\begin{equation}
x_{k+1} = f_d(x_k,u_k),
\qquad k = 0,\ldots,N-1,
\label{eq:dt_dynamics}
\end{equation}
where $N=(t_f-t_0)/\Delta t$ and $x_0=x$. A control horizon $M\le N$ is introduced such that only the inputs $\{u_k\}_{k=0}^{M-1}$ are treated as independent decision variables. For $k=M,\ldots,N-1$, the control input is held constant,
\begin{equation}
u_k = u_{M-1}, \qquad k=M,\ldots,N-1.
\label{eq:dt_blocking}
\end{equation}
The corresponding discrete-time optimal control problem is
\begin{align}
\min_{\{x_k,u_k\}} \quad
& J_d(x) =
\sum_{k=0}^{N-1} \ell(x_k,u_k)
+ \phi(x_N)
\label{eq:dt_cost} \\[0.6em]
\text{s.t.}\quad
& x_{k+1} = f_d(x_k,u_k),
\qquad k = 0,\ldots,M-1,
\label{eq:dt_ocp_dynamics_1} \\[0.2em]
& x_{k+1} = f_d(x_k,u_{M-1}),
\qquad k = M,\ldots,N-1,
\label{eq:dt_ocp_dynamics_2} \\[0.6em]
& g(x_k,u_k) \le 0,
\qquad k = 0,\ldots,M-1,
\label{eq:dt_path_constraints_1} \\[0.2em]
& g(x_k,u_{M-1}) \le 0,
\qquad k = M,\ldots,N-1,
\label{eq:dt_path_constraints_2} \\[0.6em]
& \psi(x_N) \le 0,
\label{eq:dt_terminal_constraints}
\end{align}
\subsection{Assumptions}
Throughout the paper, the following assumptions are adopted:
\begin{itemize}
\item The system dynamics are linear and time-invariant.
\item State and input constraints are linear inequalities.
\item The performance index is quadratic, with $Q \succeq 0$ and $R \succ 0$.
\item A finite time horizon is considered.
\item The control horizon satisfies $M \le N$; throughout this work $M = N$.
\end{itemize}

\subsection{Pontryagin's Minimum Principle Conditions}

The necessary conditions for optimality of the continuous-time problem are given by Pontryagin's Minimum Principle (PMP) \cite{kopp1962pontryagin,augustin2001computational}. The continuous-time optimal control problem is stated as
\begin{equation}
\min_{u(t)} \; \Phi(x(t_f)) + \int_{t_0}^{t_f} L(x(t),u(t),t)\,dt
\end{equation}
subject to
\begin{equation}
\dot{x}(t) = f(x(t),u(t),t),
\end{equation}
\begin{equation}
g_\ell(x(t),u(t),t) \le 0, \quad \ell=1,\ldots,m,
\end{equation}
\begin{equation}
\psi_j(x(t_f)) = 0, \quad j=1,\ldots,p,
\qquad x(t_0)=x_0 .
\end{equation}
The Hamiltonian function is defined as
\begin{equation}
\begin{aligned}
H(x,u,\lambda,\mu,t)
= {} & L(x,u,t)
+ \lambda^\top(t)\,f(x(t),u(t),t) \\
& + \sum_{\ell=1}^{m} \mu_\ell(t)\,g_\ell(x,u,t).
\end{aligned}
\end{equation}
The PMP conditions consist of the state equations
\begin{equation}
\dot{x}(t) = \nabla_\lambda H = f(x(t),u(t),t),
\end{equation}
the costate equations
\begin{equation}
\dot{\lambda}(t) = -\nabla_x H,
\end{equation}
and the optimality condition
\begin{equation}
\nabla_u H = 0.
\end{equation}
Constraints are enforced through complementary slackness,
\begin{equation}
\mu_\ell(t)\,g_\ell(x(t),u(t),t)=0,
\quad \ell=1,\ldots,m,
\end{equation}
together with feasibility conditions
\begin{equation}
\mu_\ell(t)\ge 0,
\qquad
g_\ell(x(t),u(t),t)\le 0.
\end{equation}
The terminal boundary conditions for the costate variables are given by the transversality conditions associated with terminal equality constraints \cite{stengel1994optimal},
\begin{equation}
\lambda(t_f)
=
\nabla_x \Phi\big(x(t_f)\big)
+
\sum_{j=1}^{p} \nu_j\,\nabla_x \psi_j\big(x(t_f)\big).
\end{equation}
When the activity status of constraints changes, the optimal solution may exhibit corner points at which the classical Weierstrass--Erdmann conditions apply \cite{gelfand2000calculus}. In particular, for problems with continuous dynamics and no switching costs, the costate variables and the Hamiltonian are continuous across the switching time $t_s$, leading to the jump conditions
\begin{equation}
\lambda(t_s^+) = \lambda(t_s^-),
\qquad
H(t_s^+) = H(t_s^-),
\end{equation}
where $t_s \in (t_0,t_f)$ denotes a switching point at which the activity status of one or more path constraints changes. The structure and interpretation of switching points are discussed in detail in the subsequent section.

For comparison, a discrete-time counterpart of the above problem is obtained by time discretization over a finite horizon. Let $t_k=t_0+k\Delta t$, $k=0,\ldots,N$, and define $x_k\approx x(t_k)$ and $u_k\approx u(t_k)$. The discrete-time optimal control problem is written as
\begin{equation}
\min_{\{x_k,u_k\}}
\; \Phi(x_N) + \sum_{k=0}^{N-1} L(x_k,u_k,k)
\end{equation}
subject to
\begin{equation}
x_{k+1} = f_d(x_k,u_k),
\quad k=0,\ldots,N-1,
\end{equation}
\begin{equation}
g_{\ell,k}(x_k,u_k) \le 0,
\quad \ell=1,\ldots,m,\quad k=0,\ldots,N-1,
\end{equation}
\begin{equation}
\psi_j(x_N) = 0,
\quad j=1,\ldots,p,
\qquad x_0=x_0.
\end{equation}
Unlike the continuous-time case, the discrete-time problem is finite-dimensional and its optimality conditions can be written using the Karush--Kuhn--Tucker (KKT) conditions. The associated Lagrangian function is
\begin{align}
\mathcal{L}
= {} & \Phi(x_N)
+ \sum_{k=0}^{N-1} L(x_k,u_k,k) \nonumber\\
& + \sum_{k=0}^{N-1}
\lambda_{k+1}^\top\Big(x_{k+1}-f_d(x_k,u_k)\Big) \nonumber\\
& + \sum_{k=0}^{N-1}\sum_{\ell=1}^{m}
\mu_{\ell,k}\,g_{\ell,k}(x_k,u_k)
+ \sum_{j=1}^{p}\nu_j\,\psi_j(x_N).
\end{align}
The KKT conditions consist of the stationarity conditions
\begin{equation}
\nabla_{x_k} \mathcal{L}=0,
\quad k=1,\ldots,N,
\end{equation}
\begin{equation}
\nabla_{u_k} \mathcal{L}=0,
\quad k=0,\ldots,M-1,
\end{equation}
\begin{equation}
\frac{\partial \mathcal{L}}{\partial \nu_j}=0,
\quad j=1,\ldots,p,
\end{equation}
together with feasibility of the constraints
\begin{equation}
x_{k+1}-f_d(x_k,u_k)=0,
\quad k=0,\ldots,N-1,
\end{equation}
\begin{equation}
g_{\ell,k}(x_k,u_k)\le 0,
\quad \ell=1,\ldots,m,\quad k=0,\ldots,N-1,
\end{equation}
\begin{equation}
\psi_j(x_N)=0,
\quad j=1,\ldots,p,
\end{equation}
dual feasibility
\begin{equation}
\mu_{\ell,k}\ge 0,
\quad \ell=1,\ldots,m,\quad k=0,\ldots,N-1,
\end{equation}
and complementary slackness
\begin{equation}
\mu_{\ell,k}\,g_{\ell,k}(x_k,u_k)=0,
\quad \ell=1,\ldots,m,\quad k=0,\ldots,N-1.
\end{equation}
A key distinction between the two formulations lies in the nature of their optimality conditions. In the continuous-time case, Pontryagin's Minimum Principle yields a system of ordinary differential equations with costates and multipliers evolving continuously over time. In contrast, the discrete-time formulation leads to a finite set of algebraic optimality conditions through the KKT framework. This fundamental difference between differential and algebraic conditions directly influences the structure of switching times and critical regions in continuous- and discrete-time formulations.

\subsection{Switching Points and Jump Conditions}
\label{subsec:switching_points}
In continuous-time constrained optimal control, the optimal solution typically changes its structure as constraints become active or inactive over time. We refer to time intervals where the same constraints remain active as \emph{arcs}. During each arc, the state and control trajectories follow a specific analytical form dictated by which constraints are currently enforced. When a constraint switches from active to inactive (or vice versa), the solution structure changes; these transition points in time are called \emph{switching points} or \emph{corners}.

For constrained arcs, it is useful to distinguish between the time at which a constraint becomes active and the time at which it is released. Each constrained arc is therefore characterized by two switching instants:
\begin{itemize}
    \item \textbf{Entry point ($t_{st}$):} the time at which the $s$-th constraint becomes active.
    \item \textbf{Exit point ($t_{sx}$):} the time at which the $s$-th constraint is deactivated.
\end{itemize}
Figure~\ref{fig:switching_schematic} illustrates these concepts schematically, while Table~\ref{tab:switching_multipliers} summarizes the relationship between switching points, constraint activity, and the associated Lagrange multipliers.

\begin{figure}[t]
    \centering
    \setlength{\fboxrule}{0.55pt}
    \setlength{\fboxsep}{0.5pt}
    \fbox{\includegraphics[width=0.92\linewidth]{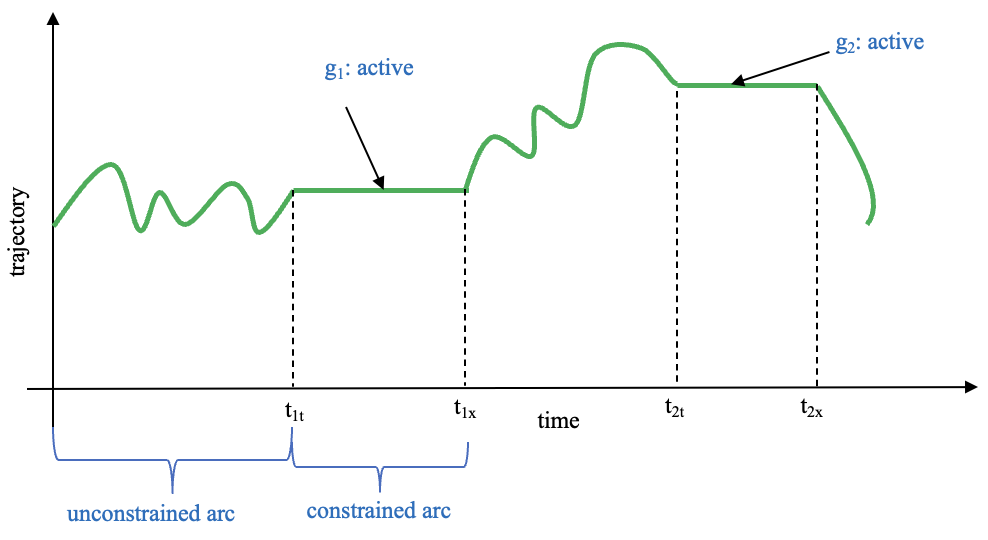}}
    \caption{Schematic illustration of switching points and constrained arcs. The times $t_{st}$ and $t_{sx}$ denote the entry and exit points of a constrained arc, respectively. Changes in the active constraint set are accompanied by corresponding changes in the associated Lagrange multipliers.}
    \label{fig:switching_schematic}
\end{figure}

\begin{table}[t]
\centering
\caption{Constraint signs and Lagrange multipliers at entry and exit switching points.}
\label{tab:switching_multipliers}
\renewcommand{\arraystretch}{1.1}
\setlength{\tabcolsep}{6pt}
\begin{tabular}{lll}
\toprule
\textbf{Corner} & \textbf{Constraint signs} & \textbf{Lagrange multipliers} \\
\midrule
\multicolumn{3}{l}{\textbf{Entry/exit of constraint $g_1$}} \\
$t_{1t}^{-}$ & $g_1 < 0,\; g_2 < 0$ & $\mu_1 = 0,\; \mu_2 = 0$ \\
$t_{1t}^{+}$ & $g_1 = 0,\; g_2 < 0$ & $\mu_1 > 0,\; \mu_2 = 0$ \\
$t_{1x}^{-}$ & $g_1 = 0,\; g_2 < 0$ & $\mu_1 > 0,\; \mu_2 = 0$ \\
$t_{1x}^{+}$ & $g_1 < 0,\; g_2 < 0$ & $\mu_1 = 0,\; \mu_2 = 0$ \\
\midrule
\multicolumn{3}{l}{\textbf{Entry/exit of constraint $g_2$}} \\
$t_{2t}^{-}$ & $g_1 < 0,\; g_2 < 0$ & $\mu_1 = 0,\; \mu_2 = 0$ \\
$t_{2t}^{+}$ & $g_2 = 0,\; g_1 < 0$ & $\mu_1 = 0,\; \mu_2 > 0$ \\
$t_{2x}^{-}$ & $g_2 = 0,\; g_1 < 0$ & $\mu_1 = 0,\; \mu_2 > 0$ \\
$t_{2x}^{+}$ & $g_1 < 0,\; g_2 < 0$ & $\mu_1 = 0,\; \mu_2 = 0$ \\
\bottomrule
\end{tabular}
\end{table}
For simplicity, when there is only one switching point we denote it by $t_s$ rather than $t_{st}$ or $t_{sx}$.

From the standpoint of the Pontryagin Minimum Principle, switching points require special treatment to ensure consistency of the optimality conditions across adjacent arcs. For problems with continuous dynamics and order-zero path constraints, two mathematically equivalent formulations can be employed. The classical approach enforces continuity of both the costate and the Hamiltonian,
\begin{equation}
\lambda(t_s^+) = \lambda(t_s^-), \qquad
H(t_s^+) = H(t_s^-),
\label{eq:jump_conditions_classical}
\end{equation}
while an alternative approach enforces continuity of the state and the costate,
\begin{equation}
x(t_s^+) = x(t_s^-), \qquad
\lambda(t_s^+) = \lambda(t_s^-).
\label{eq:jump_conditions_state}
\end{equation}
Both formulations provide two scalar conditions per switching point and lead to identical optimal solutions. However, the state-based formulation~\eqref{eq:jump_conditions_state} offers practical advantages in numerical implementations, as state variables are directly integrated from the dynamics and exhibit smoother behavior near switching points compared to the Hamiltonian, which involves composite expressions of states, costates, and multipliers.

The following proposition formalizes the connection between these two formulations and establishes conditions under which the state-based approach is sufficient.

\begin{proposition}[State-matching characterization of switching points]
\label{prop:state_matching_switch}
Consider a continuous-time optimal control problem with path constraints of order zero (depending only on $x$, $u$, and $t$). Assume continuous dynamics without impulsive terms and no switching costs. Then, for a prescribed arc sequence, the switching times can be determined by enforcing the jump conditions
\begin{equation}
x(t_s^-) = x(t_s^+), \qquad \lambda(t_s^-) = \lambda(t_s^+),
\label{eq:jump_state_costate}
\end{equation}
together with the constraint activity conditions. Under these assumptions, Hamiltonian continuity follows automatically and does not need to be imposed as an additional equation.
\end{proposition}

\begin{proof}
State continuity follows directly from the absence of impulsive dynamics. Costate continuity follows from the Weierstrass--Erdmann corner conditions, which guarantee continuous adjoint variables across switching points for order-zero constraints without switching costs.

At the switching instant, the constraint boundary is reached: $g_i(x_s, u, t_s) = 0$. Since $x_s$ and $\lambda_s$ are identical from both sides, the stationarity condition $\partial H/\partial u = 0$ together with the constraint equation determine a unique control $u_s$ under standard regularity assumptions (strict convexity of $H$ in $u$ and linear independence of active constraint gradients). Thus, $u(t_s^-) = u(t_s^+) = u_s$.

The Hamiltonian is $H = L + \lambda^\top f + \mu^\top g$. At the switching point, $g(x_s, u_s, t_s) = 0$, so the $\mu^\top g$ term vanishes regardless of $\mu$. Since $x_s$, $u_s$, and $\lambda_s$ are all continuous across $t_s$, we have $H(t_s^-) = H(t_s^+)$.
\end{proof}

\begin{remark}
\label{rem:alternative_jump}
An equivalent alternative formulation uses Hamiltonian continuity instead of state continuity:
\begin{equation}
H(t_s^-) = H(t_s^+), \qquad \lambda(t_s^-) = \lambda(t_s^+).
\end{equation}
Both formulations provide two scalar conditions per switching point. This work adopts the state-based formulation~\eqref{eq:jump_state_costate} because the state variables $x$ have direct physical meaning and yield more numerically stable computations, particularly for autonomous systems where the Hamiltonian is constant along the optimal trajectory.

When the assumptions of Proposition~\ref{prop:state_matching_switch} are violated, modified jump conditions apply:
\begin{itemize}
    \item \textbf{Higher-order constraints} (e.g., $g(\dot{x})$ or $g(\ddot{x})$): The costate exhibits a jump discontinuity given by $\Delta\lambda(t_s) = -\mu(t_s)\, \partial g/\partial \dot{x}$.
    \item \textbf{Impulsive dynamics}: The state exhibits a jump discontinuity, and the Hamiltonian-based formulation must be used instead of state matching.
    \item \textbf{Switching costs}: The Hamiltonian exhibits a jump discontinuity given by $\Delta H(t_s) = -\partial c/\partial t_s$, where $c$ denotes the switching cost function.
\end{itemize}
\end{remark}
\section{Multiparametric Continuous-Time Optimal Control}
The continuous-time optimal control problem is formulated in a multiparametric framework where parameters represent either initial states or data uncertainty within a given set. We focus on initial state parametrization, characterizing how optimal control laws, switching times, and arc structures vary with parameters.

Rather than solving individual parameter values repeatedly, the multiparametric approach treats the entire problem family in a unified framework. Each admissible parameter value yields an optimal solution satisfying Pontryagin's optimality conditions. We aim to understand how solutions evolve with parameter variations and identify parameter space regions where solution structure remains qualitatively unchanged.

In continuous-time systems, parameter variations directly affect the active constraint sequence and switching times, creating piecewise optimal solutions whose structure depends on parameter values. By analyzing this dependence explicitly, we can express optimal control laws and state trajectories as explicit parameter functions over parameter subspace regions.

This development extends classical continuous-time optimal control theory through the multiparametric framework, providing a foundation for critical region analysis and explicit solution representations.

\subsection{Parametric Formulation}
In the multiparametric setting, we treat the continuous-time optimal control problem as a family indexed by parameter vector $\theta$, studying how optimal solutions vary across an admissible set rather than solving for fixed parameter values. Throughout this work, $\theta$ represents the initial state of the system, denoted as
\begin{equation}
x(t_0) = \theta \equiv x_0, \qquad \theta \in \Theta \subset \mathbb{R}^n,
\end{equation}
where $\Theta$ denotes a compact set of admissible initial conditions.

The parametric problem then takes the form
\begin{align}
\min_{x(\cdot),\,u(\cdot)} \quad
& J(x,u;\theta) =
\int_{t_0}^{t_f} \ell\big(x(t),u(t)\big)\,dt
+ \phi\big(x(t_f)\big) \\
\text{s.t.}\quad
& \dot{x}(t) = f\big(x(t),u(t)\big), \\
& g\big(x(t),u(t)\big) \le 0, \quad \forall\, t\in[t_0,t_f], \\
& \psi\big(x(t_f)\big) \le 0, \\
& x(t_0) = \theta, \qquad \theta \in \Theta.
\end{align}

For each $\theta \in \Theta$, this formulation yields a standard continuous-time optimal control problem satisfying Pontryagin's optimality conditions. The multiparametric framework simply provides a unified way to represent the entire family of problems across different parameter values without changing the underlying dynamics or constraints.

\subsection{Parametric Switching Structures}
\label{subsec:parametric_switching}
We now examine how the structure of the optimal solution, particularly the switching instants, depends on the initial state $\theta = x_0$. All inequality constraints are written in the compact form
\begin{equation}
\label{eq:path_constraint}
g(x(t),u(t)) \le 0,
\end{equation}
with associated Lagrange multipliers satisfying
\begin{equation}
\label{eq:mu_nonneg}
\mu(t) \ge 0.
\end{equation}

Switching occurs exclusively due to changes in constraint activity. A switching instant is defined as a time at which the active set of constraints changes, or equivalently when at least one Lagrange multiplier transitions between zero and a positive value. The time horizon can therefore be decomposed as
\begin{equation}
t_0 < t_1 < \dots < t_{N_s} < t_f,
\end{equation}
where the number of switching instants $N_s$ depends on the initial state $x_0$. Each switching instant is treated as a function of the initial state,
\begin{equation}
t_s = t_s(x_0), \qquad s=1,\dots,N_s.
\end{equation}

At a switching instant $t_s$, the jump conditions for path-constrained problems without impulsive effects require continuity of the state and adjoint variables,
\begin{equation}
\label{eq:x_jump}
x(t_s^+) = x(t_s^-),
\end{equation}
\begin{equation}
\label{eq:lambda_jump}
\lambda(t_s^+) = \lambda(t_s^-).
\end{equation}

The junction conditions at a switching instant are expressed through constraint activity and the corresponding multiplier behavior. Entry of the constraint associated with switching event $s$ is characterized by
\begin{equation}
\label{eq:entry_pattern}
g(t_s^-) < 0,\ \mu(t_s^-) = 0,
\qquad
g(t_s^+) = 0,\ \mu(t_s^+) > 0,
\end{equation}
whereas exit of the constraint is characterized by
\begin{equation}
\label{eq:exit_pattern}
g(t_s^-) = 0,\ \mu(t_s^-) > 0,
\qquad
g(t_s^+) < 0,\ \mu(t_s^+) = 0.
\end{equation}
All remaining constraints satisfy feasibility and complementary slackness on both sides of the switching instant.

A parametric switching structure is defined as a fixed ordered sequence of active constraint sets over the time horizon. A critical region $\mathcal{R}\subset\Theta$ is a subset of the initial-state space such that for every $x_0\in\mathcal{R}$ the same switching structure is observed and the jump conditions \eqref{eq:x_jump}--\eqref{eq:lambda_jump} and junction conditions \eqref{eq:entry_pattern}--\eqref{eq:exit_pattern} remain valid.

Since the number and locations of switching instants are not known beforehand, a structure-driven procedure is adopted. For a fixed switching structure, exact switching times $t_s(x_0)$ are computed for selected initial states by solving the optimality system together with the jump and junction conditions. These exact values are then used to construct empirical approximations $\hat{t}_s(x_0)$ describing the dependence of switching instants on the initial state within the same structure.

The procedure is summarized in Algorithm~\ref{alg:param_switch}.

\begin{algorithm}[t]
\caption{Identification and fitting of parametric switching times}
\label{alg:param_switch}
\hrule
\vspace{0.5em}
\textbf{Input:} Switching structure $\mathcal{S}$ identified at a seed initial state $x_0^\star$, and sampled initial states $\{x_0^{(r)}\}_{r=1}^{N_{\mathcal{S}}}$ restricted to the domain where $\mathcal{S}$ is observed \\
\textbf{Output:} Empirical switching-time functions $\{\hat{t}_s(x_0)\}_{s=1}^{N_s}$ associated with $\mathcal{S}$ \\
\textbf{Indexing:} $r$ indexes sampled initial states; $s$ indexes switching events
\vspace{0.5em}
\hrule
\begin{enumerate}
\item Solve the continuous-time optimality conditions at $x_0^\star$ to identify the switching structure $\mathcal{S}$ and the ordered switching events $s=1,\dots,N_s$.
\item For each sample $r=1,\dots,N_{\mathcal{S}}$:
\begin{enumerate}
\item Solve the optimality system consistent with $\mathcal{S}$ at $x_0^{(r)}$.
\item For each switching event $s=1,\dots,N_s$, compute the exact switching time $t_s^{(r)} = t_s(x_0^{(r)})$ by enforcing the jump conditions \eqref{eq:x_jump}--\eqref{eq:lambda_jump} and the junction conditions \eqref{eq:entry_pattern} or \eqref{eq:exit_pattern}.
\item Retain $t_s^{(r)}$ if all referenced conditions are satisfied.
\end{enumerate}
\item For each switching event $s=1,\dots,N_s$, fit an empirical function $\hat{t}_s(x_0)$ using the retained sample pairs $\{(x_0^{(r)},t_s^{(r)})\}$.
\item \textbf{return} $\{\hat{t}_s(x_0)\}_{s=1}^{N_s}$.
\end{enumerate}
\hrule
\end{algorithm}
\begin{lemma}[Existence, uniqueness, and continuity of switching time]
\label{lem:switching_continuity}
For a fixed switching structure $\mathcal{S}$ with a single switching event, there exists a unique switching time $t_s(x_0) \in (t_0, t_f)$ for each $x_0 \in \mathcal{R}$, and the switching time function $t_s: \mathcal{R} \to [t_0, t_f]$ is continuous over the interior of the critical region $\mathcal{R}$.
\end{lemma}

\begin{proof}
\textbf{Step 1: Existence.}
For each $x_0 \in \mathcal{R}$, the switching time $t_s(x_0)$ is defined as the time when the Lagrange multiplier reaches zero:
\begin{equation}
\mu(t_s, x_0) = 0.
\end{equation}
For linear-quadratic problems, the multiplier $\mu(t, x_0)$ is obtained by solving the coupled system of linear ordinary differential equations governing the state and costate, with initial conditions determined by $x_0$. Since the critical region $\mathcal{R}$ is defined by a fixed switching structure, a switching time exists for each $x_0 \in \mathcal{R}$ by construction.

\textbf{Step 2: Form of the multiplier.}
For linear systems, the state and costate evolve according to the Hamiltonian system, which is a $2n$-dimensional linear ordinary differential equation. The multiplier $\mu(t, x_0)$ is determined algebraically from the stationarity condition and can therefore be expressed as a linear combination of the fundamental solutions:
\begin{equation}
\mu(t, x_0) = \sum_{i=1}^{2n} \alpha_i(x_0) e^{\sigma_i t},
\end{equation}
where $\sigma_i$ are the $2n$ eigenvalues of the Hamiltonian system matrix and $\alpha_i(x_0)$ are coefficients that depend continuously on the initial state $x_0$. This exponential form follows directly from the linearity of the state-costate system.

\textbf{Step 3: Uniqueness within the critical region.}
Within a critical region $\mathcal{R}$, the switching structure is fixed by definition. For a single switching event, the constraint transitions from active to inactive (or vice versa) exactly once.

Suppose $\mu(t, x_0) = 0$ had multiple solutions in $(t_0, t_f)$. Then the multiplier would cross zero multiple times, implying multiple constraint activations or deactivations. This would contradict the assumption of a fixed switching structure with a single switching event.

Therefore, for each $x_0 \in \mathcal{R}$, there exists a unique $t_s(x_0) \in (t_0, t_f)$ satisfying $\mu(t_s(x_0), x_0) = 0$.

\textbf{Step 4: Non-vanishing time derivative.}
At the switching time $t_s$, the constraint transitions between active and inactive states. We show that $\frac{\partial \mu}{\partial t}(t_s, x_0) \neq 0$.

From Step~2, we have $\mu(t, x_0) = \sum_{i=1}^{2n} \alpha_i(x_0) e^{\sigma_i t}$, which is smooth and has a well-defined derivative at every point. Suppose, for the sake of contradiction, that $\frac{\partial \mu}{\partial t}(t_s, x_0) = 0$. Since $\mu(t_s, x_0) = 0$ by definition of the switching time, this gives:
\begin{equation}
\mu(t_s) = 0 \quad \text{and} \quad \frac{d\mu}{dt}(t_s) = 0.
\end{equation}

Consider the Taylor expansion of $\mu$ around $t_s$:
\begin{equation}
\mu(t) = \mu(t_s) + \frac{d\mu}{dt}(t_s)(t - t_s) + O\big((t - t_s)^2\big) = O\big((t - t_s)^2\big).
\end{equation}
This implies that $\mu(t)$ remains near zero with second-order smallness in a neighborhood of $t_s$. However, within the critical region $\mathcal{R}$, the switching structure requires that on one side of $t_s$ the constraint is active (with $\mu > 0$), and on the other side it is inactive (with $\mu = 0$ by complementarity).

If $\frac{d\mu}{dt}(t_s) = 0$, then $\mu$ cannot transition immediately from a positive value to zero. It would instead hover near zero, requiring higher-order derivatives to vanish as well. But from the exponential form $\mu(t, x_0) = \sum_{i=1}^{2n} \alpha_i(x_0) e^{\sigma_i t}$, if $\mu$ and all its derivatives vanish at $t_s$, then $\mu \equiv 0$ for all $t$, contradicting the existence of an active constraint phase.

Therefore, $\frac{\partial \mu}{\partial t}(t_s, x_0) \neq 0$.

\textbf{Step 5: Continuity.}
Since $\mu(t, x_0)$ is continuously differentiable in both arguments and $\frac{\partial \mu}{\partial t}(t_s, x_0) \neq 0$, the Implicit Function Theorem \cite{luenberger1997optimization} guarantees that there exists a unique continuous function $t_s(x_0)$ in a neighborhood of any $x_0^* \in \mathrm{int}(\mathcal{R})$ satisfying $\mu(t_s(x_0), x_0) = 0$. Since this holds for every point in the interior of $\mathcal{R}$, the switching time function $t_s: \mathcal{R} \to [t_0, t_f]$ is continuous over $\mathrm{int}(\mathcal{R})$.
\end{proof}

\begin{remark}
\label{rem:weierstrass}
The continuity established in Lemma~\ref{lem:switching_continuity} has important practical implications for the representation of switching-time functions. By the Weierstrass Approximation Theorem, any continuous function on a compact set can be uniformly approximated by polynomials to arbitrary precision. Since critical regions are bounded subsets of the parameter space and $t_s(x_0)$ is continuous over each region, polynomial approximations are theoretically guaranteed to converge to the true switching-time function.

When an analytical expression is available, it should be used directly. When analytical derivation is intractable, the Weierstrass theorem justifies the empirical fitting procedure in Algorithm~\ref{alg:param_switch}. In practice, low-degree polynomials achieve high accuracy over the relatively narrow parameter ranges of individual critical regions. Alternative approximation schemes---such as rational functions, splines, or other basis functions---may be employed depending on the complexity of the switching-time dependence and the desired balance between accuracy and computational simplicity.
\end{remark}
\subsection{Critical Regions and Explicit Solutions}
\label{subsec:critical_regions}

The parametric switching structures identified in the previous subsection induce a partition of the initial-state space into subsets within which the qualitative structure of the optimal solution remains unchanged. These subsets are referred to as \emph{critical regions} and provide the basis for constructing explicit parametric solutions.

For a fixed switching structure, a critical region is the set of initial states $x_0 \in \Theta$ for which the same constraints remain active and inactive over the entire time horizon. In practice, this is enforced by two simple criteria that mirror the discrete-time definition of critical regions (for example, in multiparametric quadratic programming or explicit MPC), with the only difference being that the conditions must hold over a continuous-time horizon:
\begin{itemize}
    \item \textbf{Inactive constraints remain inactive.}
    \item \textbf{Active Lagrange multipliers remain positive.}
\end{itemize}

Accordingly, a critical region $\mathcal{R}$ is characterized by parametric inequalities evaluated along the optimal trajectory. Let $\tilde{g}(\cdot)$ denote the vector of inactive inequality constraints, and let $\tilde{\mu}(\cdot)$ and $\tilde{\nu}(\cdot)$ denote the Lagrange multipliers associated with active constraints and their entry points, respectively. The interior of a critical region is defined by
\begin{align}
\tilde{g}\!\big(x(t,t_s(x_0),x_0),u(t,t_s(x_0),x_0),x_0\big) &< 0, \label{eq:cr_inactive}\\
\tilde{\mu}\!\big(t,t_s(x_0),x_0\big) &> 0, \label{eq:cr_active_mu}\\
\tilde{\nu}\!\big(t_s(x_0),x_0\big) &> 0. \label{eq:cr_entry}
\end{align}

The inequalities \eqref{eq:cr_inactive}--\eqref{eq:cr_entry} are time-dependent. However, to obtain an explicit description of a critical region in the initial-state space, it is sufficient to enforce them at their most restrictive points in time. This yields explicit parametric expressions that define the region boundaries.

For each inactive constraint $i$, the boundary is given by
\begin{equation}
\bar{G}_i(x_0) =
\max_{t \in [t_0,t_f]}
\tilde{g}_i\!\big(x,u,x_0),x_0\big),
\label{eq:inactive_boundary}
\end{equation}
and the condition $\bar{G}_i(x_0) < 0$ defines the interior of the region with respect to that constraint.

For constraints that are active over at least one constrained arc $[t_{i,s}^-(x_0),t_{i,s}^+(x_0)]$, the inactive portions of the horizon must also satisfy
\begin{equation}
\bar{G}_i(x_0) =
\max_{\substack{t \in [t_0,t_f] \\ t \notin [t_{i,s}^-(x_0),t_{i,s}^+(x_0)]}}
\tilde{g}_i\!\big(x,u,x_0\big),
\label{eq:mixed_boundary}
\end{equation}

Finally, for each active constraint $i$, positivity of the associated Lagrange multiplier at all entry points must be preserved, which yields the boundary condition
\begin{equation}
\bar{\mu}_i(x_0) =
\min_{\substack{t \in [t_{i,s}^-(x_0),t_{i,s}^+(x_0)]}}\, \mu_i\!\big(x,u,x_0\big), 
\label{eq:multiplier_boundary}
\end{equation}

Equations \eqref{eq:inactive_boundary}--\eqref{eq:multiplier_boundary} provide explicit parametric expressions for the \emph{boundaries} of the critical regions. Crossing any of these boundaries implies that at least one previously inactive constraint becomes active (or an active multiplier loses positivity), which changes the switching structure and therefore leads to a different explicit solution.

Within a critical region, the switching structure is fixed and the optimality system admits a uniform representation. As a result, the optimal control, state, and adjoint trajectories can be expressed explicitly as functions of time and the initial state. The optimal control law takes the piecewise form
\begin{equation}
u^\ast(t,x_0)=
\begin{cases}
u_1(t,x_0), & t \in [t_0,t_1(x_0)),\\
u_2(t,x_0), & t \in [t_1(x_0),t_2(x_0)),\\
\vdots & \\
u_{N_s+1}(t,x_0), & t \in [t_{N_s}(x_0),t_f],
\end{cases}
\end{equation}
where each arc-wise expression is obtained from the Hamiltonian minimization condition under the fixed active set associated with the region.

\begin{algorithm}[t]
\caption{Solution of the multiparametric dynamic optimization problem}
\label{alg:mpdo}
\hrule
\vspace{0.5em}
\textbf{Input:} Parameter domain $\Theta$, initial parameter seed $\theta^{(0)} \in \Theta$ \\
\textbf{Output:} Set of critical regions and exact parametric solutions
\vspace{0.5em}
\hrule
\begin{enumerate}
\item Initialize the unexplored parameter set $\Theta_{\mathrm{rest}} \leftarrow \Theta$.
\item \textbf{While} $\Theta_{\mathrm{rest}} \neq \emptyset$ \textbf{do}:
\begin{enumerate}
\item Select a fixed parameter point $\theta^{(i)} \in \Theta_{\mathrm{rest}}$.
\item Solve the dynamic optimization problem at $\theta^{(i)}$ to identify the optimal arc structure, active constraints, and switching points.
\item Derive exact analytical expressions for the state, adjoint, control, and multipliers corresponding to this structure.
\item Construct symbolic boundaries of the critical regions by enforcing feasibility, complementarity, and multiplier sign conditions using \eqref{eq:inactive_boundary}--\eqref{eq:multiplier_boundary}.
\item Define the associated critical region $\mathcal{R}_i \subset \Theta$ as the set of parameters satisfying these conditions.
\item Remove the identified region from the unexplored set: $\Theta_{\mathrm{rest}} \leftarrow \Theta_{\mathrm{rest}} \setminus \mathcal{R}_i$.
\end{enumerate}
\item \textbf{return} the collection of critical regions and their exact parametric solutions.
\end{enumerate}
\hrule
\end{algorithm}

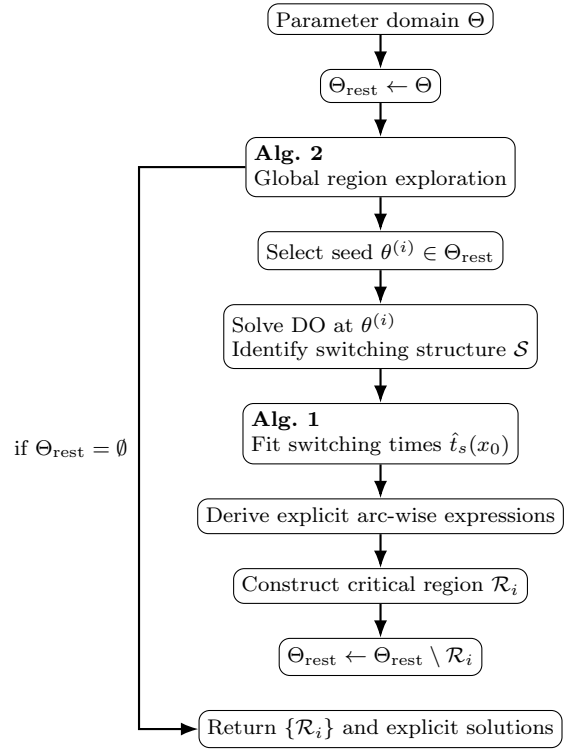
\begin{figure}[t]
\centering
\scriptsize
\setlength{\tabcolsep}{3pt}
\begin{tikzpicture}[
  node distance=4.5mm and 9mm,
  box/.style={rectangle, draw, rounded corners, align=left, inner sep=3pt},
  startstop/.style={rectangle, draw, rounded corners, align=center, inner sep=3pt},
  arrow/.style={-Latex, thick}
]

\node[startstop] (start) {Parameter domain $\Theta$};
\node[box, below=of start] (init) {$\Theta_{\mathrm{rest}} \leftarrow \Theta$};

\node[box, below=of init] (outer) {\textbf{Alg.~\ref{alg:mpdo}}\\
Global region exploration};

\node[box, below=of outer] (seed) {Select seed $\theta^{(i)} \in \Theta_{\mathrm{rest}}$};
\node[box, below=of seed] (solve) {Solve DO at $\theta^{(i)}$\\
Identify switching structure $\mathcal{S}$};

\node[box, below=of solve] (inner) {\textbf{Alg.~\ref{alg:param_switch}}\\
Fit switching times $\hat{t}_s(x_0)$};

\node[box, below=of inner] (explicit) {Derive explicit arc-wise expressions};
\node[box, below=of explicit] (region) {Construct critical region $\mathcal{R}_i$};
\node[box, below=of region] (remove) {$\Theta_{\mathrm{rest}} \leftarrow \Theta_{\mathrm{rest}}\setminus\mathcal{R}_i$};

\node[startstop, below=of remove] (end) {Return $\{\mathcal{R}_i\}$ and explicit solutions};

\draw[arrow] (start) -- (init);
\draw[arrow] (init) -- (outer);
\draw[arrow] (outer) -- (seed);
\draw[arrow] (seed) -- (solve);
\draw[arrow] (solve) -- (inner);
\draw[arrow] (inner) -- (explicit);
\draw[arrow] (explicit) -- (region);
\draw[arrow] (region) -- (remove);

\draw[arrow] (outer.west) -- ++(-14mm,0) |- 
node[pos=0.25, left, xshift=-1mm, fill=white, inner sep=1pt]
{if $\Theta_{\mathrm{rest}}=\emptyset$} (end.west);

\end{tikzpicture}
\caption{Interaction between Alg.~\ref{alg:mpdo} (global exploration) and Alg.~\ref{alg:param_switch} (local switching-time fitting).}
\label{fig:alg_interaction}
\end{figure}
Figure~\ref{fig:alg_interaction} shows how the two algorithms interact: Algorithm~\ref{alg:mpdo} drives the global exploration, calling Algorithm~\ref{alg:param_switch} locally to fit switching times for each identified structure.

\section{Illustrative Examples}

This section presents two illustrative examples to demonstrate the formulation and solution of multiparametric dynamic optimization problems in both continuous-time (CT) and discrete-time (DT) settings. For each example, the continuous-time problem is first solved analytically, leading to an explicit partition of the parameter space into critical regions with region-specific optimal trajectories. Once the critical region corresponding to a given initial condition is identified, the optimal state and control trajectories are obtained by direct evaluation of the associated expressions, with switching handled through region-dependent switching times when applicable.

All continuous-time computations are performed symbolically using \texttt{SymPy}, with switching-time expressions evaluated via \texttt{lambdify}. When closed-form expressions are not retained, sampled solutions are post-processed using \texttt{NumPy} to construct low-degree polynomial approximations for reporting. The discrete-time counterparts are formulated as multiparametric quadratic programs and solved using the \texttt{PPOPT} toolbox \cite{kenefake2022ppopt}. Specifically, we employ the combinatorial parallel algorithm implemented in \texttt{PPOPT}, which systematically explores the space of active sets to obtain the corresponding critical-region partitions and affine control laws.

\subsection{Example 1: One-State Integrating System with Mixed Output and Input Constraints}
\label{subsec:ex1}

This example illustrates the formulation and solution of a multiparametric dynamic optimization problem for a simple continuous-time system with path constraints.

\subsubsection{Problem Statement}
\label{subsubsec:ex1_problem}

\begin{equation}
\label{eq:ex1_ocp}
\begin{aligned}
\min_{u(\cdot)}\quad
& \frac{1}{2}x(T)^2 + \frac{1}{2}\int_{0}^{T}\Big(x(t)^2 + u(t)^2\Big)\,dt
\\
\text{s.t.}\quad
& \dot{x}(t) = -u(t), \qquad t\in[0,T],
\\
& y(t) = x(t) + u(t), \qquad t\in[0,T],
\\
& -1 \le y(t) \le 1, \qquad t\in[0,T],
\\
& u(t) \le 2, \qquad t\in[0,T],
\\
& x(0)=x_0,\qquad -2 \le x_0 \le 2,
\\
& T=2.
\end{aligned}
\end{equation}
The initial state $x_0$ is the parameter, and the aim is to derive an explicit representation of the optimal control law over $\Theta$. The construction is presented below.
\subsubsection{Continuous-Time Multiparametric Solution}
\label{subsubsec:ex1_ct_mpdo}

The initial state $x_0$ is treated as the scalar parameter on the domain
\[
\Theta := \{x_0\in\mathbb{R}\mid -2\le x_0 \le 2\}.
\]
The objective is to derive an explicit parametric description of the optimal trajectories and their switching structure over $\Theta$.

\medskip
\noindent
\emph{Iteration 1.}

\begin{enumerate}
\item[(1)] Initialize the unexplored parameter set:
\[
\Theta_{\mathrm{rest}} \leftarrow \Theta.
\]
\item[(2a)] Fix a seed parameter in the unexplored set:
\[
x_0^{(0)} := 0 \in \Theta_{\mathrm{rest}}.
\]
\item[(2b)] At the seed point $x_0^{(0)}=0$, the optimality conditions 
yield a trajectory that satisfies all path and input 
constraints throughout $[0,2]$, so no switching points 
arise and the solution is a single unconstrained arc.
\item[(2c)] Derive the exact analytical expressions for the state, control, adjoint, and constraint-related quantities corresponding to this arc. The Hamiltonian is
\[
H=\tfrac{1}{2}(x^2+u^2)+\lambda(-u),
\]
which gives the stationarity condition $u(t)=\lambda(t)$. The state and adjoint equations are
\[
\dot{x}(t)=-u(t), \qquad \dot{\lambda}(t)=-x(t),
\]
with terminal condition $\lambda(2)=x(2)$. Solving yields
\[
x(t)=x_0 e^{-t},\,
u(t)=x_0 e^{-t},\,
\lambda(t)=x_0 e^{-t},\, t\in[0,2].
\]
The path constraint function is
\[
y(t)=x(t)+u(t),
\]
and since the constraint is inactive on this arc, the corresponding multiplier satisfies
\[
\mu(t)=0 \quad \forall t\in[0,2].
\]
\item[(2d)] Enforce feasibility conditions symbolically over the horizon. The output is
\[
y(t)=x(t)+u(t)=2x_0 e^{-t}.
\]
Requiring $-1\le y(t)\le 1$ for all $t\in[0,2]$ gives
\[
-1 \le 2x_0 e^{-t} \le 1.
\]
Since $e^{-t}\in(0,1]$, feasibility reduces to
\[
-1 \le 2x_0 \le 1 \quad \Longleftrightarrow \quad -0.5 \le x_0 \le 0.5.
\]
\item[(2e)] Define the associated critical region:
\[
\mathrm{CR03}_{\mathrm{CT}} = \{x_0\in\Theta \mid -0.5 \le x_0 \le 0.5\},
\]
with a single unconstrained arc on $[0,2]$.
\item[(2f)] Update the unexplored parameter set:
\[
\Theta_{\mathrm{rest}} \leftarrow [-2,-0.5)\cup(0.5,2].
\]
\item[(3)] Since $\Theta_{\mathrm{rest}}\neq\emptyset$, proceed to the next iteration.
\end{enumerate}

\medskip
\noindent
\emph{Iteration 2.}

\begin{enumerate}
\item[(2a)] Fix a new seed parameter in the unexplored set:
\[
x_0^{(1)} := -0.8 \in \Theta_{\mathrm{rest}}.
\]
\item[(2b)] At $x_0^{(1)} = -0.8$, the optimality conditions show that the unconstrained solution violates the lower output bound at $t=0$. The optimal trajectory therefore has two arcs: 
an initial constrained arc with the lower bound active, followed by an unconstrained arc after the constraint is released at a single switching point.
\item[(2c)] Derive exact analytical expressions for the state, control, and constraint-related quantities for the identified two-arc structure. Define the lower output constraint as
\[
g(t) = -x(t)-u(t)-1 \leq 0.
\]
On the first arc the constraint is active, hence
\[
g(t)=0, \qquad \mu(t)\geq 0.
\]
which implies
\[
u(t)=-1-x(t), \qquad \dot{x}(t)=1+x(t).
\]
With $x(0)=x_0$, the constrained-arc solution is
\[
x_0(t)=(x_0+1)e^{t}-1, \,
u_0(t)=-(x_0+1)e^{t}, \, t\in[0,t_s].
\]
On the second arc the constraint is inactive, so that
\[
g(t)\leq 0, \qquad \mu(t)=0.
\]
The unconstrained optimal feedback applies,
\[
u(t)=x(t), \qquad \dot{x}(t)=-x(t),
\]
yielding, by continuity at $t_s$,
\[
x_1(t)=x(t_s)e^{-(t-t_s)}, \,
u_1(t)=x(t_s)e^{-(t-t_s)}, \, t\in[t_s,2].
\]
\item[(2d)] Enforce the symbolic boundary and junction conditions. The switching time $t_s$ corresponds to leaving the active constraint. At this point the trajectory must satisfy the constraint boundary and be compatible with the unconstrained feedback law.

On the unconstrained arc, $u(t)=x(t)$ implies
\[
g(t)=-x(t)-u(t)-1=-2x(t)-1.
\]
Enforcing $g(t_s)=0$ therefore yields
\[
-2x(t_s)-1=0
\quad\Longrightarrow\quad
x(t_s)=-\frac{1}{2}.
\]
Substituting the constrained-arc expression gives
\[
(x_0+1)e^{t_s}-1=-\frac{1}{2}
\quad\Longrightarrow\quad
t_s(x_0)=\ln\!\Big(\frac{1}{2(x_0+1)}\Big).
\]
Requiring $t_s(x_0)\in[0,2]$ yields
\[
-1+\frac{1}{2e^2} \le x_0 \le -\frac{1}{2}.
\]

\item[(2e)] Define the associated critical region:
\[
\mathrm{CR02}_{\mathrm{CT}} =
\left\{x_0\in\Theta \,\middle|\,
-1+\frac{1}{2e^2} \le x_0 \le -0.5
\right\}.
\]
\item[(2f)] Remove this region from the unexplored set. The remaining portion of the negative side is
\[
[-2,-0.5]\setminus \mathrm{CR02}_{\mathrm{CT}} = [-2,-1+\tfrac{1}{2e^2}).
\]
\item[(3)] At this stage, two critical regions have been constructed explicitly: $\mathrm{CR03}_{\mathrm{CT}}$ and $\mathrm{CR02}_{\mathrm{CT}}$.
\end{enumerate}
By symmetry, the positive half of the parameter space yields analogous regions with opposite signs. The complete procedure identifies five critical regions. The active-arc structures associated with the continuous-time critical regions are summarized in Table~\ref{tab:ct_active_sets_all}.

\begin{table}[t]
\centering
\caption{Active-arc structures of the continuous-time critical regions and their corresponding parameter intervals.}
\label{tab:ct_active_sets_all}
\begin{tabular}{lll}
\hline
Region & $x_0$ interval & Arc structure \\
\hline
CR01$_{\mathrm{CT}}$ & $[-1.2707,\,-0.9323]$ & $y_{\min}$ active (full horizon) \\
CR02$_{\mathrm{CT}}$ & $[-0.9323,\,-0.5]$ & $y_{\min}$ active $\rightarrow$ unconstrained \\
CR03$_{\mathrm{CT}}$ & $[-0.5,\,0.5]$ & unconstrained (full horizon) \\
CR04$_{\mathrm{CT}}$ & $[0.5000,\,0.9323]$ & $y_{\max}$ active $\rightarrow$ unconstrained \\
CR05$_{\mathrm{CT}}$ & $[0.9323,\,2]$ & $y_{\max}$ active (full horizon) \\
\hline
\end{tabular}
\end{table}
These structures identify which constraints are active over the time horizon for each parameter interval. Given the active-arc structure of each region, the corresponding optimal trajectories can be derived in closed form. The resulting closed-form expressions for the optimal state and control trajectories are reported in Table~\ref{tab:CT_explicit}.

\newcolumntype{Y}{>{\raggedright\arraybackslash}X}

\begin{table}[t]
\centering
\caption{Explicit control and state functions for continuous-time regions.}
\label{tab:CT_explicit}

\small                    
\setlength{\tabcolsep}{3pt}
\renewcommand{\arraystretch}{1.1}

\begin{tabularx}{\columnwidth}{@{} l c Y @{}}
\toprule
Region & Inequality & Control and state functions \\
\midrule

CR01$_{\mathrm{CT}}$
& $-1.2707 \le x_0 \le -0.9323$
& $0 \le t \le T$\\
&& $u(t)=(-x_0-1)e^{t}$\\
&& $x(t)=(x_0+1)e^{t}-1$
\\[1mm]

CR02$_{\mathrm{CT}}$
& $-0.9323 \le x_0 \le -0.5$
& $0 \le t \le t_s$\\
&& $u_0(t)=(-x_0-1)e^{t}$\\
&& $x_0(t)=(x_0+1)e^{t}-1$\\
&& $t_s \le t \le T$\\
&& $u_1(t)=-\tfrac{1}{2}e^{-t+t_s}$\\
&& $x_1(t)=-\tfrac{1}{2}e^{-t+t_s}$
\\[1mm]

CR03$_{\mathrm{CT}}$
& $-0.5 \le x_0 \le 0.5$
& $0 \le t \le T$\\
&& $u(t)=x_0 e^{-t}$\\
&& $x(t)=x_0 e^{-t}$
\\[1mm]

CR04$_{\mathrm{CT}}$
& $0.5 \le x_0 \le 0.9323$
& $0 \le t \le t_s$\\
&& $u_0(t)=(1-x_0)e^{t}$\\
&& $x_0(t)=(x_0-1)e^{t}+1$\\
&& $t_s \le t \le T$\\
&& $u_1(t)=\tfrac{1}{2}e^{-t+t_s}$\\
&& $x_1(t)=\tfrac{1}{2}e^{-t+t_s}$
\\[1mm]

CR05$_{\mathrm{CT}}$
& $0.9323 \le x_0 \le 2$
& $0 \le t \le T$\\
&& $u(t)=(1-x_0)e^{t}$\\
&& $x(t)=(x_0-1)e^{t}+1$
\\

\bottomrule
\end{tabularx}
\end{table}

\medskip
\noindent
\emph{Switching-time identification on $\mathrm{CR02}_{\mathrm{CT}}$.}

The switching-time function $t_s(x_0)$ associated with $\mathrm{CR02}_{\mathrm{CT}}$ is identified following the procedure outlined in Algorithm~\ref{alg:param_switch}. Although an analytical expression for the switching time can be derived for this example, such a closed-form relationship is not generally available in multiparametric dynamic optimization problems. In more complex systems, switching instants may be defined implicitly through nonlinear junction and jump conditions, making direct analytical derivation infeasible. For this reason, a sampling-based identification and fitting strategy is employed.

\begin{enumerate}
\item[(1)] Identify the switching structure at a representative seed. A representative initial state $x_0^\star\in\mathrm{CR02}_{\mathrm{CT}}$ is selected and the continuous-time optimality conditions are solved. The resulting optimal trajectory exhibits a single switching event corresponding to the release of the lower output constraint. The identified switching structure is therefore
\[
\mathcal{S}:\quad y_{\min}\ \text{active}\ \rightarrow\ \text{unconstrained},
\]
with one switching event ($N_s=1$).
\item[(2)] Sample initial states within the domain where the same switching structure is observed. A finite set of initial states
\[
\{x_0^{(r)}\}_{r=1}^{N_R} \subset \mathrm{CR02}_{\mathrm{CT}}
\]
is selected such that the switching structure is preserved for all samples.
\item[(2a)] Solve the optimality system consistent with the identified structure. For each sample $x_0^{(r)}$, the dynamic optimization problem is solved by enforcing the same two-arc structure as in $\mathrm{CR02}_{\mathrm{CT}}$: an initial arc with active lower output constraint followed by an unconstrained arc. The state, control, adjoint, and multiplier trajectories are computed accordingly.
\item[(2b)] Compute the switching time for each sample by enforcing junction and jump conditions. For each $x_0^{(r)}$, the switching time $t_s^{(r)}$ is obtained by enforcing the junction condition associated with leaving the active constraint. In this example, the condition reduces to
\[
x\big(t_s^{(r)}\big)=-0.5,
\]
which yields
\[
t_s^{(r)}=\ln\!\Big(\frac{1}{2(x_0^{(r)}+1)}\Big).
\]
\item[(2c)] Retain only switching times consistent with feasibility and structural assumptions. The pair $\big(x_0^{(r)},t_s^{(r)}\big)$ is retained only if the switching time satisfies $t_s^{(r)}\in[0,2]$ and the constraint activity before and after the switching instant is consistent with the identified structure $\mathcal{S}$.
\item[(3)] Fit an empirical switching-time function. Using the retained sample pairs $\{(x_0^{(r)},t_s^{(r)})\}$, an empirical approximation of the switching-time function is constructed. A cubic polynomial is found to provide an accurate representation over $\mathrm{CR02}_{\mathrm{CT}}$:
\[
\hat{t}_s(x_0)= -25.63\,x_0^3 -46.14\,x_0^2 -30.03\,x_0 -6.71,
\]
with coefficient of determination $R^2=0.9990$.
\item[(4)] Return the switching-time function associated with the structure. The fitted function $\hat{t}_s(x_0)$ is associated with structure 
$\mathcal{S}$ on $\mathrm{CR02}_{\mathrm{CT}}$. The exact expression 
is available here for reference:
\[
t_s(x_0)=\ln\!\Big(\frac{1}{2(x_0+1)}\Big).
\]
\end{enumerate}
The resulting exact and fitted switching-time functions obtained from the above procedure are summarized in Table~\ref{tab:ct_ts_all}.
\begin{table}[t]
\centering
\caption{Switching-time functions (exact and fitted) for switching regions.}
\label{tab:ct_ts_all}

\footnotesize
\setlength{\tabcolsep}{3pt}
\renewcommand{\arraystretch}{0.8}

\begin{tabular}{@{} 
L{0.14\columnwidth} 
L{0.25\columnwidth} 
L{0.43\columnwidth} 
L{0.11\columnwidth} 
@{}}
\toprule
Region & Exact $t_s(x_0)$ & Fitted $\hat t_s(x_0)$ & $R^2$ \\
\midrule

CR02$_{\mathrm{CT}}$
& $\ln\!\big(\frac{1}{2(x_0+1)}\big)$
& $\hat t_s(x_0)=
\mminus 25.6336x_0^3
\mminus 46.1437x_0^2
\mminus 30.0310x_0
\mminus 6.7059$
& 0.999037 \\

CR04$_{\mathrm{CT}}$
& $\ln\!\big(\frac{-1}{2(x_0-1)}\big)$
& $\hat t_s(x_0)=
\pplus 25.6336x_0^3
\mminus 46.1437x_0^2
\pplus 30.0310x_0
\mminus 6.7059$
& 0.999037 \\

\bottomrule
\end{tabular}
\end{table}
For each switching region, the exact expression is available for reference, while the fitted function provides a practical approximation when an explicit analytical form is not available.

\subsubsection{Discrete-Time Counterpart}
\label{subsubsec:ex1_dt}

To compare with the continuous-time multiparametric solution, the problem is discretized on a uniform grid $t_k=kh$, $k=0,\dots,N$, with step size $h=T/N$ and a zero-order hold (ZOH) input parameterization,
\[
u(t)=u_k,\qquad t\in[t_k,t_{k+1}),\quad k=0,\dots,N-1.
\]
Under ZOH, the continuous-time dynamics $\dot{x}(t)=A x(t)+B u(t)$ admit the exact discrete-time representation
\begin{equation}
x_{k+1}=A_d x_k + B_d u_k,
\label{eq:ex1_dt_dyn_general}
\end{equation}
where $A_d=e^{Ah}$ and $B_d=\int_{0}^{h}e^{A\tau}B\,d\tau$, which is the standard mapping used in MPC for linear systems \cite{rawlings2020model}.

For the specific system $\dot{x}(t)=-u(t)$, we have $A=0$ and $B=-1$, hence
\begin{equation}
x_{k+1}=x_k-hu_k,\qquad k=0,\dots,N-1,
\label{eq:ex1_dt_dyn}
\end{equation}
which is exact under the ZOH input parameterization.

The continuous-time constraints
\[
y(t)=x(t)+u(t),\qquad -1\le y(t)\le 1,\qquad u(t)\le 2
\]
are enforced at the grid nodes only,
\begin{equation}
-1\le x_k+u_k\le 1,\qquad u_k\le 2,\qquad k=0,\dots,N-1,
\label{eq:ex1_dt_cons}
\end{equation}
together with the initial condition $x_0=\theta$, $\theta\in[-2,2]$.

The continuous-time cost
\[
J=\frac{1}{2}x(T)^2+\frac{1}{2}\int_0^T\!\left(x(t)^2+u(t)^2\right)\,dt
\]
is approximated by a forward-Euler quadrature,
\begin{equation}
J_N
=
\frac{1}{2}x_N^2
+\frac{1}{2}\sum_{k=0}^{N-1}h\left(x_k^2+u_k^2\right),
\label{eq:ex1_dt_cost_euler}
\end{equation}
which yields the discrete-time parametric quadratic program
\begin{equation}
\begin{aligned}
\min_{\{x_k,u_k\}} \quad & J_N \\
\text{s.t.}\quad
& x_{k+1}=x_k-hu_k,\qquad k=0,\dots,N-1,\\
& -1\le x_k+u_k\le 1,\qquad k=0,\dots,N-1,\\
& u_k\le 2,\qquad k=0,\dots,N-1,\\
& x_0=\theta,\qquad \theta\in[-2,2].
\end{aligned}
\label{eq:ex1_dt_mpqp}
\end{equation}

The explicit multiparametric solution (critical-region partition and affine laws) is computed for \eqref{eq:ex1_dt_mpqp} using \texttt{PPOPT}. For $N=5$, the parameter space is partitioned into 11 discrete-time critical regions. Table~\ref{tab:ex1_dt_crs} reports three representative regions from the complete set of 11.

\begin{table}[t]
\centering
\caption{Representative explicit DT critical regions for $N=5$ ($h=0.4$):
affine laws $u(\theta)=K_u\theta+k_u$ and $x(\theta)=K_x\theta+k_x$.}
\label{tab:ex1_dt_crs}

\scriptsize
\setlength{\tabcolsep}{3.5pt} 
\renewcommand{\arraystretch}{0.92}

\begin{tabular}{@{}c@{\hspace{0.5pt}}c@{\hspace{2pt}}p{0.34\columnwidth}p{0.34\columnwidth}@{}}
\toprule
Region & $\theta$ interval & $u(\theta)$ & $x(\theta)$ \\
\midrule

$\mathrm{CR}^{\mathrm{DT}}_{01}$ &
$[-1.52,\,-0.89]$ &
\scalebox{0.85}{$
\begin{aligned}
\begin{bmatrix}-1\\-1.4\\-1.96\\-2.744\\-3.8416\end{bmatrix}\theta
&+\begin{bmatrix}-1\\-1.4\\-1.96\\-2.744\\-3.8416\end{bmatrix}
\end{aligned}$}
&
\scalebox{0.85}{$
\begin{aligned}
\begin{bmatrix}1.4\\1.96\\2.744\\3.8416\\5.37824\end{bmatrix}\theta
&+\begin{bmatrix}0.4\\0.96\\1.744\\2.8416\\4.37824\end{bmatrix}
\end{aligned}$}
\\[6pt]

$\mathrm{CR}^{\mathrm{DT}}_{06}$ &
$[-0.55,\,0.55]$ &
\scalebox{0.9}{$
\begin{bmatrix}0.815\\0.546\\0.363\\0.239\\0.153\end{bmatrix}\theta$}
&
\scalebox{0.9}{$
\begin{bmatrix}0.674\\0.456\\0.310\\0.215\\0.153\end{bmatrix}\theta$}
\\[6pt]

$\mathrm{CR}^{\mathrm{DT}}_{11}$ &
$[0.89,\,2]$ &
\scalebox{0.85}{$
\begin{aligned}
\begin{bmatrix}-1\\-1.4\\-1.96\\-2.744\\-3.8416\end{bmatrix}\theta
&+\begin{bmatrix}1\\1.4\\1.96\\2.744\\3.8416\end{bmatrix}
\end{aligned}$}
&
\scalebox{0.85}{$
\begin{aligned}
\begin{bmatrix}1.4\\1.96\\2.744\\3.8416\\5.37824\end{bmatrix}\theta
&-\begin{bmatrix}0.4\\0.96\\1.744\\2.8416\\4.37824\end{bmatrix}
\end{aligned}$}
\\

\bottomrule
\end{tabular}
\end{table}

Within each discrete-time critical region, the optimal control and state trajectories are given by affine functions of $\theta$. The larger number of DT regions compared to the continuous-time partition reflects the growth of admissible KKT-active sets induced by time discretization and node-wise constraint enforcement.

To illustrate the impact of discretization on explicit-solution complexity, Table~\ref{tab:dt_cr_vs_N} reports the number of DT critical regions for increasing $N$.

\begin{table}[t]
\centering
\caption{Number of DT critical regions versus discretization level $N$.}
\label{tab:dt_cr_vs_N}
\scriptsize
\setlength{\tabcolsep}{6pt}
\renewcommand{\arraystretch}{1.05}
\begin{tabular}{cc}
\toprule
$N$ & Number of DT critical regions \\
\midrule
5  & 11 \\
10 & 21 \\
15 & 31 \\
30 & 61 \\
\bottomrule
\end{tabular}
\end{table}

Table~\ref{tab:dt_cr_vs_N} shows that the number of DT critical regions increases with $N$, consistent with the combinatorial growth in admissible active-constraint patterns as the number of grid nodes increases.
\subsection{Example 2: Two-State System with Mixed Output and Input Constraints}
This example extends the framework to a two-state system with path constraints, demonstrating the approach on a more complex dynamic system.

\subsubsection{Problem Statement}

\begin{equation}
\label{eq:ex2_matrix}
\begin{aligned}
\min_{u(\cdot)}\quad
& \frac{1}{2}\,x(T)^\top x(T)
+ \frac{1}{2}\int_{0}^{T}
\Big( x(t)^\top x(t) + u(t)^2 \Big)\,dt
\\
\text{s.t.}\quad
& \dot{x}(t) =
\begin{bmatrix}
0 & -1 \\
-1 & 0
\end{bmatrix}
x(t)
+
\begin{bmatrix}
1 \\ 0
\end{bmatrix}
u(t),
\qquad t \in [0,T],
\\
& y(t) =
\begin{bmatrix}
-1 & 1 \\
\;\;1 & -1
\end{bmatrix}
x(t)
+
\begin{bmatrix}
1 \\ -1
\end{bmatrix}
u(t),
\qquad t \in [0,T],
\\
& y(t) \le
\begin{bmatrix}
1.2 \\ 2
\end{bmatrix},
\qquad t \in [0,T],
\\
& x(0) = x_0,
\qquad
-2 \le x_{0,i} \le 2,\quad i=1,2,
\\
& T = 2 .
\end{aligned}
\end{equation}
The initial state $(x_{0,1},x_{0,2})$ is treated as the parameter in problem~\eqref{eq:ex2_matrix}, and the objective is to express the optimal control law explicitly as a function of the initial state, subject to path constraints enforced over the full horizon.
\subsubsection{Continuous-Time Multiparametric Solution}
\label{subsubsec:ex2_ct_solution}

The parameter domain is
\[
\Theta := \{(x_{0,1},x_{0,2})\in\mathbb{R}^2 \mid -2\le x_{0,1}\le 2,\ -2\le x_{0,2}\le 2\}.
\]

\noindent
(1) Initialize the unexplored parameter set:
\[
\Theta_{\mathrm{rest}} \leftarrow \Theta.
\]

\noindent
(2a) Fix a seed point in the initial set:
\[
x_0^{(2)}=\begin{bmatrix}0\\0.5\end{bmatrix}\in\Theta_{\mathrm{rest}}.
\]

\noindent
(2b) Solve the dynamic optimization problem at $x_0=x_0^{(2)}$. The optimal structure consists of two arcs,
\[
\text{Arc A (active)}:\ 0\le t\le t_s,
\,
\text{Arc B (inactive)}:\ t_s\le t\le T,
\]
where the first path constraint is active on Arc A,
\[
g_1(t)=y_1(t)-1.2=0,
\qquad \mu_1(t)\ge 0,
\]
and Arc B is unconstrained with $\mu_1(t)=\mu_2(t)=0$.

\noindent
(2c) Closed-form expressions on each arc.

\textbf{Arc B (inactive, $t_s\le t\le T$).}

On this arc, $\mu_1(t)=\mu_2(t)=0$. The stationarity condition gives $u(t)=-\lambda_1(t)$. Substituting into the state-costate equations yields the linear homogeneous system
\[
\frac{d}{dt}
\begin{bmatrix}
x_1(t)\\
x_2(t)\\
\lambda_1(t)\\
\lambda_2(t)
\end{bmatrix}
=
A_{\mathrm{in}}
\begin{bmatrix}
x_1(t)\\
x_2(t)\\
\lambda_1(t)\\
\lambda_2(t)
\end{bmatrix},
\,
A_{\mathrm{in}}=
\begin{bmatrix}
0 & -1 & -1 & 0\\
-1& 0  & 0  & 0\\
-1& 0  & 0  & 1\\
0 & -1 & 1  & 0
\end{bmatrix}.
\]

A closed-form solution is
\[
\begin{aligned}
x_1(t) &= \sqrt{2}\,C_1 e^{-\sqrt{2}\,t}-\sqrt{2}\,C_2 e^{\sqrt{2}\,t}+C_3 e^{-t}+C_4 e^{t},\\
x_2(t) &= C_1 e^{-\sqrt{2}\,t}+C_2 e^{\sqrt{2}\,t}+C_3 e^{-t}-C_4 e^{t},\\
\lambda_1(t) &= C_1 e^{-\sqrt{2}\,t}+C_2 e^{\sqrt{2}\,t},\\
\lambda_2(t) &= C_3 e^{-t}+C_4 e^{t},
\qquad t\in[t_s,T].
\end{aligned}
\]
where $C_1,\ldots,C_4$ are arc constants.

\textbf{Arc A (constraint 1 active, $0\le t\le t_s$).}

On this arc, $g_1(t)=0$ and $\mu_1(t)\ge 0$. Using stationarity and differentiating the active constraint leads to an augmented linear system in $(x,\lambda,\mu_1)$,
\[
\frac{d}{dt}
\begin{bmatrix}
x_1(t)\\
x_2(t)\\
\lambda_1(t)\\
\lambda_2(t)\\
\mu_1(t)
\end{bmatrix}
=
A_{\mathrm{ac}}
\begin{bmatrix}
x_1(t)\\
x_2(t)\\
\lambda_1(t)\\
\lambda_2(t)\\
\mu_1(t)
\end{bmatrix},
\,
A_{\mathrm{ac}}=
\begin{bmatrix}
0 & -1 & -1 & 0 & -1\\
-1& 0  & 0  & 0 & 0\\
-1& 0  & 0  & 1 & 1\\
0 & -1 & 1  & 0 & -1\\
0 & 1  & 1  & -1& 0
\end{bmatrix}.
\]

A closed-form solution is
\[
\begin{aligned}
x_1(t) &= C'_1 e^{-t}-\frac{4}{3}C'_2 e^{2t},\\
x_2(t) &= C'_1 e^{-t}+\frac{2}{3}C'_2 e^{2t}-C'_3,\\
\lambda_1(t) &= C'_2 e^{2t}-C'_4 e^{-2t}-C'_5 e^{t},\\
\lambda_2(t) &= C'_1 e^{-t}-\frac{1}{3}C'_2 e^{2t}-C'_3 + C'_4 e^{-2t}-2C'_5 e^{t},\\
\mu_1(t) &= C'_2 e^{2t}+C'_3+C'_4 e^{-2t}+C'_5 e^{t},
\qquad t\in[0,t_s].
\end{aligned}
\]
where $C'_1,\ldots,C'_5$ are arc constants.

To determine the coefficients in these closed-form expressions, the boundary conditions at $t=0$ and $t=T$, together with the jump and junction conditions at $t=t_s$, are imposed on the two-arc solution.

(i) Boundary conditions:
\[
x_1(0)=x_{0,1},\, x_2(0)=x_{0,2},\,
\lambda_1(T)=x_1(T),\, \lambda_2(T)=x_2(T).
\]

(ii) Jump (matching) conditions at the switching time:
\[
x(t_s^-)=x(t_s^+),
\qquad
\lambda(t_s^-)=\lambda(t_s^+).
\]

(iii) Junction condition at the switching time:
\[
\mu_1(t_s^+)=0,
\qquad
\mu_1(t_s^-)\ge 0.
\]

After applying these conditions and simplifying, the following explicit relations are obtained:
\[
C'_3=-0.6,
\]
\[
C'_1 = \frac{x_{0,1}+2x_{0,2}+2C'_3}{3},
\qquad
C'_2 = \frac{x_{0,2}-x_{0,1}+C'_3}{2},
\]
and the inactive-arc constants $C_1$ and $C_2$ can be written as linear functions of $(C_3,C_4)$,
\[
C_1 = 169.21C_4  -0.81 C_3,
\qquad
C_2 = 0.28C_4 +0.0028C_3.
\]

For the remaining coefficients, the explicit expressions become lengthy. In general, these coefficients depend on $(x_{0,1},x_{0,2})$ and include exponential terms involving the switching time $t_s(x_{0,1},x_{0,2})$. For compact reporting, we approximate the longest expressions with fitted forms that preserve this dependence. Table~\ref{tab:ex2_fit_coeffs} lists the fitted coefficients used in $\mathrm{CR02}_{\mathrm{CT}}$.

\begin{table}[t]
\centering
\footnotesize
\setlength{\tabcolsep}{3pt}
\renewcommand{\arraystretch}{1.15}
\caption{Selected fitted coefficients used for reporting in $\mathrm{CR02}_{\mathrm{CT}}$.}
\label{tab:ex2_fit_coeffs}

\begin{tabular}{@{} c p{0.74\columnwidth} c @{}}
\toprule
Coefficient & Fit (in terms of $x_{0,1},x_{0,2},t_s$) & $R^2$\\
\midrule

$C_4$
& \adjustbox{max width=\linewidth}{$-0.033\,x_{0,1}e^{-0.88 t_s}+0.038\,x_{0,2}e^{-0.73 t_s}-0.013$}
& 0.997 \\

$C'_4$
& \adjustbox{max width=\linewidth}{$-1.73\,x_{0,1}e^{1.61 t_s}+1.72\,x_{0,2}e^{1.62 t_s}$}
& 0.999 \\

$C'_5$
& \adjustbox{max width=\linewidth}{$0.022\,x_{0,1}e^{-0.88 t_s}-0.025\,x_{0,2}e^{-0.73 t_s}+0.009$}
& 0.997 \\

$C_3$
& \adjustbox{max width=\linewidth}{$-2.12\,x_{0,1}e^{0.87 t_s}+3.07\,x_{0,2}e^{0.68 t_s}$}
& 0.999 \\
\bottomrule
\end{tabular}
\end{table}

Moreover, the switching time $t_s$ is determined as part of the algebraic solution associated with the two-arc structure. For parameter values where switching occurs, the function $t_s(x_{0,1},x_{0,2})$ is obtained by sampling the parameter space and solving the dynamic optimization problem. The resulting fitted switching-time functions are reported in Table~\ref{tab:ex2_ts_fits}.

\begin{table}[t]
\centering
\caption{Fitted switching-time functions for the CT switching regions in Example~2.}
\label{tab:ex2_ts_fits}

\footnotesize
\setlength{\tabcolsep}{3pt}
\renewcommand{\arraystretch}{1.2}

\begin{tabular}{@{} c L{0.72\columnwidth} c @{}}
\toprule
Region & Fitted $\hat t_s(x_{0,1},x_{0,2})$ & $R^2$ \\
\midrule

CR02$_{\mathrm{CT}}$
& $\hat t_s=
-1.10
\mminus 4.49x_{0,1}
\pplus 4.48x_{0,2}
\mminus 4.68x_{0,1}^2
\pplus 9.35x_{0,1}x_{0,2}
\mminus 4.67x_{0,2}^2
\mminus 2.17x_{0,1}^3
\pplus 6.49x_{0,1}^2x_{0,2}
\mminus 6.49x_{0,1}x_{0,2}^2
\pplus 2.16x_{0,2}^3$
& 0.999 \\

CR04$_{\mathrm{CT}}$
& $\hat t_s=
-1.13
\pplus 2.79x_{0,1}
\mminus 2.79x_{0,2}
\mminus 1.81x_{0,1}^2
\pplus 3.63x_{0,1}x_{0,2}
\mminus 1.82x_{0,2}^2
\pplus 0.52x_{0,1}^3
\mminus 1.57x_{0,1}^2x_{0,2}
\pplus 1.57x_{0,1}x_{0,2}^2
\mminus 0.52x_{0,2}^3$
& 0.999 \\

\bottomrule
\end{tabular}
\end{table}

The resulting continuous-time critical regions are summarized in Table~\ref{tab:ex2_ct_arc_regions}, together with their corresponding arc structures and defining inequalities in the parameter space.

\begin{table}[t]
\centering
\footnotesize
\caption{Active-arc structures and critical-region descriptions for Example~2.}
\label{tab:ex2_ct_arc_regions}

\setlength{\tabcolsep}{3pt}
\renewcommand{\arraystretch}{1.15}

\begin{tabular}{@{} C{0.14\columnwidth} C{0.26\columnwidth} L{0.56\columnwidth} @{}}
\toprule
\textbf{Region} & \textbf{Arc structure} & \textbf{Region description (borders)}\\
\midrule

CR01$_{\mathrm{CT}}$ &
{\scriptsize unconstrained} &
\parbox[t]{\linewidth}{\footnotesize
$\{-3.36x_{0,1}+3.36x_{0,2}-1.2\le 0,$\\
\hspace*{1.2em} $\,3.36x_{0,1}-3.36x_{0,2}-2\le 0\}$}
\\

CR02$_{\mathrm{CT}}$ &
{\scriptsize $y_1$ active $\rightarrow$ unconstrained} &
\parbox[t]{\linewidth}{\footnotesize
$\{-3.36x_{0,1}+3.36x_{0,2}-1.2\ge 0,$\\
\hspace*{1.2em} $-1.00x_{0,1}+x_{0,2}-0.61\le 0\}$}
\\

CR03$_{\mathrm{CT}}$ &
{\scriptsize $y_1$ active (full)} &
$\{-1.00x_{0,1}+x_{0,2}-0.61\ge 0\}$
\\

CR04$_{\mathrm{CT}}$ &
{\scriptsize $y_2$ active $\rightarrow$ unconstrained} &
\parbox[t]{\linewidth}{\footnotesize
$\{3.36x_{0,1}-3.36x_{0,2}-2\ge 0,$\\
\hspace*{1.2em} $1.00x_{0,1}-x_{0,2}-1.01\le 0\}$}
\\

CR05$_{\mathrm{CT}}$ &
{\scriptsize $y_2$ active (full)} &
$\{1.00x_{0,1}-x_{0,2}-1.01\ge 0\}$
\\

\bottomrule
\end{tabular}
\end{table}

A graphical representation of the resulting critical-region partition is shown in Figure~\ref{fig:ex2_ct_cr_map}.

\begin{figure}[t]
    \centering
    \includegraphics[width=0.7\linewidth]{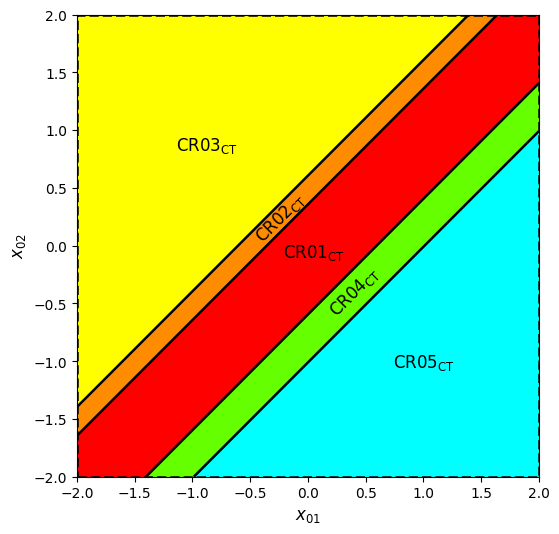}
    \caption{Continuous-time critical-region map for Example 2.}
    \label{fig:ex2_ct_cr_map}
\end{figure}

Each colored region in Figure~\ref{fig:ex2_ct_cr_map} corresponds to one of the active-arc structures listed in Table~\ref{tab:ex2_ct_arc_regions}.

\subsubsection{Discrete-Time Counterpart}
\label{subsubsec:ex2_dt}

The discrete-time counterpart of Example 2 is obtained using the same modeling and discretization choices adopted in Example 1. In the following, we report discrete-time critical-region maps for two discretization levels and provide representative affine laws for one region.

\begin{figure}[t]
    \centering
    \begin{minipage}[t]{0.7\linewidth}
        \centering
        \includegraphics[width=\linewidth]{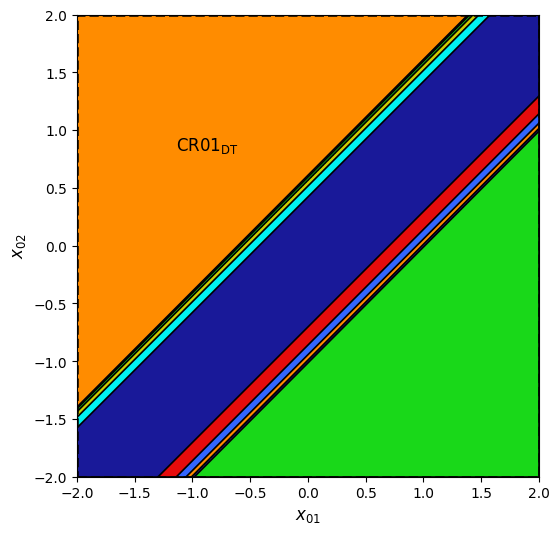}
        \vspace{1pt}
        {\small (a) $N=5$}
    \end{minipage}\hfill
    \begin{minipage}[t]{0.7\linewidth}
        \centering
        \includegraphics[width=\linewidth]{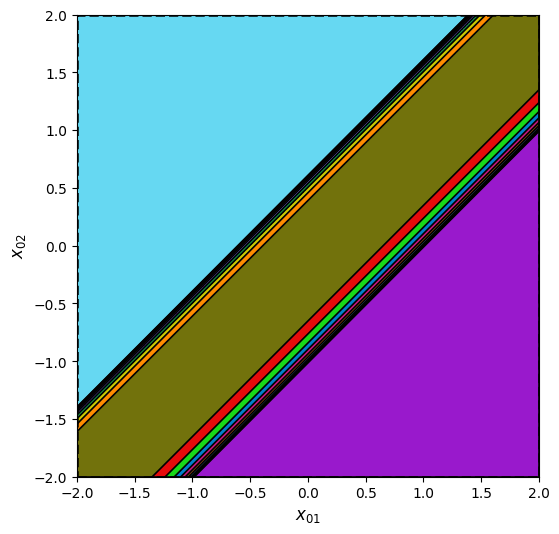}
        \vspace{1pt}
        {\small (b) $N=10$}
    \end{minipage}
    \caption{Discrete-time critical-region maps for Example 2: (a) $N=5$ (11 critical regions); (b) $N=10$ (21 critical regions).}
    \label{fig:ex2_dt_cr_maps}
\end{figure}

Table~\ref{tab:ex2_dt_cr01_N5} reports the affine state and control laws for a representative region, $\mathrm{CR01}_{\mathrm{DT}}$, obtained for $N=5$.

\begin{table}[t]
\centering
\footnotesize
\caption{Representative discrete-time critical region for Example 2 ($\mathrm{CR01}_{\mathrm{DT}}$, $N=5$).}
\label{tab:ex2_dt_cr01_N5}
\setlength{\tabcolsep}{4pt}
\renewcommand{\arraystretch}{1.15}
\begin{tabular}{c c l}
\hline
\textbf{Variable} & \textbf{Index} & \textbf{Affine law} \\
\hline
$u(\theta)$ & $k=0$ & $u_0 = 1.00\,x_{0,1}-1.00\,x_{0,2}+1.20$ \\
 & $k=1$ & $u_1 = 1.98\,x_{0,1}-1.98\,x_{0,2}+1.79$ \\
 & $k=2$ & $u_2 = 3.93\,x_{0,1}-3.93\,x_{0,2}+2.96$ \\
 & $k=3$ & $u_3 = 7.81\,x_{0,1}-7.81\,x_{0,2}+5.28$ \\
 & $k=4$ & $u_4 = 15.48\,x_{0,1}-15.48\,x_{0,2}+9.89$ \\[3mm]
\hline
$x(\theta)$& $k=1$ &
$\begin{aligned}
x_{1,1} &= 1.49\,x_{0,1}-0.82\,x_{0,2}+0.49\\
x_{1,2} &= -0.49\,x_{0,1}+1.16\,x_{0,2}-0.10
\end{aligned}$ \\[3mm]
 & $k=2$ &
$\begin{aligned}
x_{2,1} &= 2.63\,x_{0,1}-2.18\,x_{0,2}+1.31\\
x_{2,2} &= -1.31\,x_{0,1}+1.75\,x_{0,2}-0.45
\end{aligned}$ \\[3mm]
 & $k=3$ &
$\begin{aligned}
x_{3,1} &= 5.00\,x_{0,1}-4.69\,x_{0,2}+2.82\\
x_{3,2} &= -2.81\,x_{0,1}+3.11\,x_{0,2}-1.27
\end{aligned}$ \\[3mm]
 & $k=4$ &
$\begin{aligned}
x_{4,1} &= 9.76\,x_{0,1}-9.56\,x_{0,2}+5.74\\
x_{4,2} &= -5.72\,x_{0,1}+5.92\,x_{0,2}-2.95
\end{aligned}$ \\[3mm]
 & $k=5$ &
$\begin{aligned}
x_{5,1} &= 19.26\,x_{0,1}-19.13\,x_{0,2}+11.48\\
x_{5,2} &= -11.45\,x_{0,1}+11.59\,x_{0,2}-6.35
\end{aligned}$ \\
\hline
\end{tabular}
\end{table}

Finally, Table~\ref{tab:ex2_dt_cr_vs_N} summarizes the number of discrete-time critical regions obtained for increasing discretization levels.

\begin{table}[t]
\centering
\caption{Number of DT critical regions versus discretization level $N$ for Example 2.}
\label{tab:ex2_dt_cr_vs_N}
\setlength{\tabcolsep}{8pt}
\renewcommand{\arraystretch}{1.15}
\begin{tabular}{c c}
\hline
$N$ & Number of DT critical regions \\
\hline
5  & 11 \\
8 & 17 \\
12 & 25 \\
20 & 41 \\
\hline
\end{tabular}
\end{table}

\subsection{Comparison and Remarks}
\label{subsubsec:comparison_ct_dt}

Figure~\ref{fig:control_comparison_ex2} reveals the essential distinction between the two approaches through a representative trajectory from Example~2. The continuous-time solution exhibits a smooth exponential profile with a switching time at $t_s = 0.1396$. With only $N=5$ intervals, the discrete-time approximation produces a crude staircase trajectory that misses both the switching instant and the exponential structure. Refining to $N=20$ improves accuracy considerably, but the number of critical regions grows from 11 to 41, whereas the continuous-time partition remains fixed at 5 regions. This reveals a fundamental trade-off: discrete-time formulations must balance approximation accuracy against partition complexity.

\begin{figure}[t]
    \centering
    \includegraphics[width=0.85\linewidth]{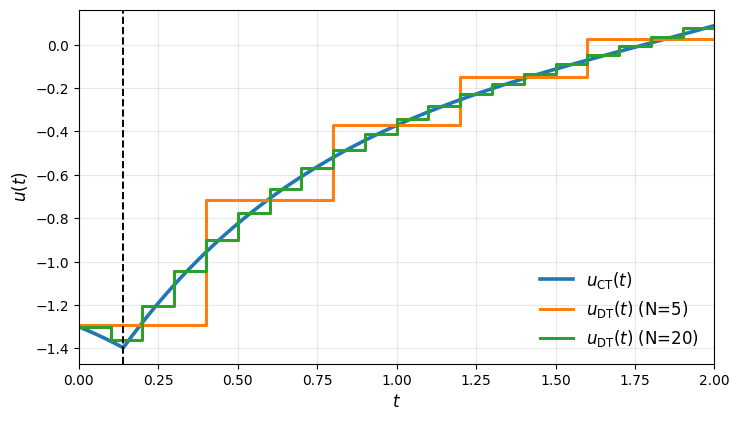}
    \caption{Optimal control trajectories for Example~2 at $x_0 = (-0.95, -1.65)$: continuous-time (blue) versus discrete-time with $N=5$ (orange) and $N=20$ (green).}
    \label{fig:control_comparison_ex2}
\end{figure}

This visual difference stems from how each approach handles control parameterization and time. Discrete-time formulations adopt a zero-order hold assumption, restricting the control to be piecewise constant across $N$ intervals. This introduces $N$ state and control variables as decision variables, leading to Karush--Kuhn--Tucker conditions that form a finite-dimensional algebraic system scaling with $N$. Continuous-time formulations impose no such restriction—time remains a continuous variable, and the optimal control structure emerges from Pontryagin's Maximum Principle through coupled state--adjoint differential equations. The problem dimension stays fixed, determined only by the system order.

The resulting solutions take distinctly different forms. For linear--quadratic problems, continuous-time solutions are closed-form exponential functions of time, explicitly depending on $x_0$, parameter-dependent switching times $t_s(x_0)$, and continuous time $t$ (Tables~\ref{tab:CT_explicit} and Section~\ref{subsubsec:ex2_ct_solution}). Discrete-time solutions are piecewise affine in $\theta$ at each grid point $k$, taking the form $u(\theta)=K_u\theta+k_u$ and $x(\theta)=K_x\theta+k_x$ (Tables~\ref{tab:ex1_dt_crs} and~\ref{tab:ex2_dt_cr01_N5}), with time entering implicitly through the index rather than explicitly.

A particularly valuable feature of the continuous-time approach is its explicit identification of switching times $t_s(x_0)$, revealing precisely when the active constraint set changes as a function of initial conditions. Table~\ref{tab:ct_ts_all} provides analytical expressions characterizing these transitions. Discrete-time formulations cannot capture this structure directly—transitions occur only at predefined grid points, making the apparent switching behavior dependent on discretization rather than system dynamics.

The partition complexity difference is substantial. For Example~1, the discrete-time formulation requires 11 regions at $N=5$ and 61 regions at $N=30$ (Tables~\ref{tab:dt_cr_vs_N} and~\ref{tab:ex2_dt_cr_vs_N}), while the continuous-time partition contains just 5 regions regardless of discretization. This combinatorial growth arises from the increasing number of constraint activation patterns across more time nodes, forcing practitioners to balance accuracy against computational burden.

Constraint enforcement reveals another key difference. Continuous-time formulations enforce path constraints over the entire horizon, for all $t \in [0,T]$, yielding a fixed feasible region. Discrete-time formulations check constraints only at grid points, potentially missing inter-sample violations. In Example~1, the continuous-time feasible set is $\theta \in [-1.27, 2.0]$, whereas discrete-time with $N=5$ admits $\theta \in [-1.52, 2.0]$. As $N$ increases, the discrete-time feasible region contracts toward its continuous-time counterpart.

Despite these differences, both formulations partition the parameter space using the same principle: within each critical region, the active constraint set remains invariant and the corresponding Lagrange multipliers satisfy strict complementarity. The distinction lies in representation—continuous-time solutions encode this through switching times and exponential trajectories, while discrete-time solutions approximate it through an expanding collection of affine segments.

\section{Discussion}

The comparison in Section~4 shows that continuous-time formulations achieve an order-of-magnitude reduction in critical regions—5 versus 61 in Example~1. This directly impacts online implementation: locating the current parameter among fewer regions is substantially faster, determining practical viability for systems requiring rapid control updates or operating under strict computational constraints.

The examples focus on linear-quadratic problems, which arise in process control, trajectory planning, and resource allocation. For this class, Pontryagin's Maximum Principle yields analytical solutions. Several directions warrant investigation for extending continuous-time multiparametric methods beyond this setting.

When uncertainty such as disturbances affects the system, robust continuous-time multiparametric formulations become necessary to maintain feasibility and constraint satisfaction. Developing systematic approaches for handling parametric uncertainty in the continuous-time framework represents an important research direction.

Computational tools specifically designed for solving linear-quadratic continuous-time multiparametric optimization problems do not yet exist. Developing such software that automates the workflow from problem specification to parametric solution construction would facilitate practical implementation and wider adoption.

Problems with multiple switching times introduce significant complexity. When the number of switches varies with parameters, critical regions may become nonconvex. Additionally, when the minimum over adjoint multipliers or maximum over path constraint violations do not occur at time endpoints (0 or $T$), the critical region boundaries become nonlinear rather than linear. Systematic methods for handling these cases while preserving computational advantages merit investigation.

Extending to nonlinear systems presents both challenges and opportunities. Recent results \cite{10886652} demonstrate that strict convexity of the Hamiltonian in the control variable ensures global optimality for continuous-time nonlinear optimal control problems, even when dynamics and costs are nonconvex. This contrasts with discrete-time formulations, where Karush-Kuhn-Tucker conditions provide only necessary conditions and multiple local optima may exist. For nonlinear systems where the Hamiltonian is convex in control, Pontryagin's Maximum Principle yields globally optimal solutions—an advantage over discrete-time approaches. However, the resulting boundary value problems rarely admit closed-form solutions, necessitating hybrid numerical-explicit approaches. Identifying which nonlinear classes permit tractable continuous-time multiparametric solutions while preserving global optimality guarantees represents a promising research direction.

Within the PAROC framework \cite{pistikopoulos2015paroc} (Figure~\ref{fig:paroc_framework}), the continuous-time multiparametric approach developed here represents the next generation: working directly with ordinary differential equations through Pontryagin's Maximum Principle rather than algebraic Karush-Kuhn-Tucker conditions from discretization. Model reduction remains employed when needed, but the fundamental shift is from algebraic formulations to differential equation-based methods. For linear systems, this ODE-based direction is now viable; extending to the broader problem classes outlined above would establish it as a practical framework for more complex applications.

\begin{figure}[t]
    \centering
    \includegraphics[width=1.05\linewidth]{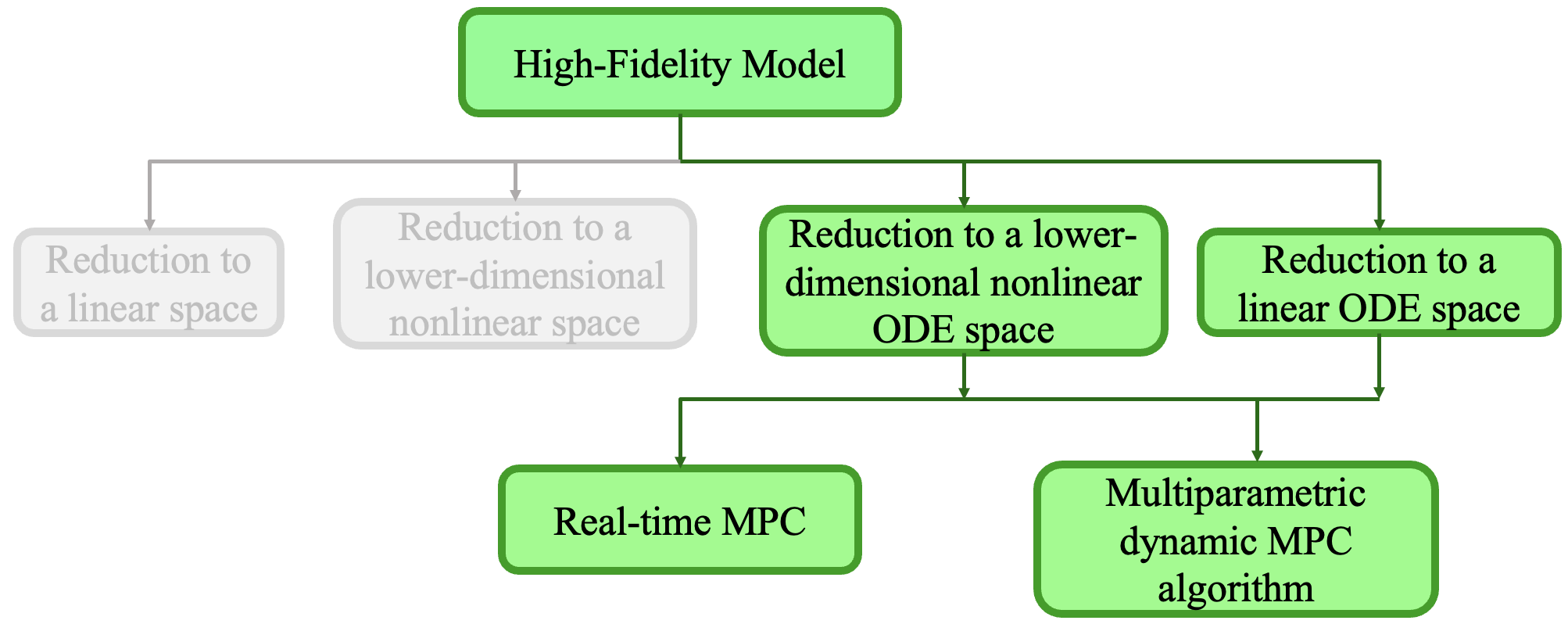}
    \caption{Next generation of parametric optimization and control (PAROC) framework. The highlighted path indicates the trajectory toward multiparametric methods in ODE spaces.}
    \label{fig:paroc_framework}
\end{figure}
\section{Conclusions}

This work demonstrates that multiparametric optimal control can be performed directly on ordinary differential equations for linear-quadratic systems, achieving order-of-magnitude reductions in solution complexity compared to discrete-time formulations. The continuous-time partition remains fixed at 5 critical regions regardless of accuracy requirements, while discrete-time partitions grow from 11 to over 60 regions as discretization refines—a difference that directly impacts both offline computation and online region lookup for real-time model predictive control.

The computational advantage stems from working directly with Pontryagin's Maximum Principle rather than discretized Karush-Kuhn-Tucker conditions. By avoiding artificial time discretization, the continuous-time approach captures the essential switching structure more compactly while providing explicit analytical expressions for switching times $t_s(x_0)$ and state-control trajectories unavailable from piecewise-affine discrete-time laws. These explicit switching structures also accelerate gradient-based online optimization by providing analytical information about solution topology.

The systematic algorithms presented here—for switching structure identification, parametric switching time computation, and critical region construction—establish continuous-time multiparametric programming as a practical computational framework for linear-quadratic systems. Natural extensions include robust formulations under parametric uncertainty, nonlinear systems where Hamiltonian convexity ensures global optimality despite nonconvex dynamics, and dedicated computational tools that automate the workflow from problem specification to parametric solution construction. By demonstrating viability for the linear-quadratic case, this work opens a path toward the next generation of parametric methods operating directly in ordinary differential equation spaces.

\vspace{6pt}

\bibliographystyle{plain}
\bibliography{references}

@book{rawlings2020model,
  title={Model predictive control: theory, computation, and design},
  author={Rawlings, James Blake and Mayne, David Q and Diehl, Moritz and others},
  volume={2},
  year={2020},
  publisher={Nob Hill Publishing Madison, WI}
}

@article{muske1993model,
  title={Model predictive control with linear models},
  author={Muske, Kenneth R and Rawlings, James B},
  journal={AIChE Journal},
  volume={39},
  number={2},
  pages={262--287},
  year={1993},
  publisher={Wiley Online Library}
}

@article{richalet1993industrial,
  title={Industrial applications of model based predictive control},
  author={Richalet, Jacques},
  journal={Automatica},
  volume={29},
  number={5},
  pages={1251--1274},
  year={1993},
  publisher={Elsevier}
}

@article{pistikopoulos2002line,
  title={On-line optimization via off-line parametric optimization tools},
  author={Pistikopoulos, Efstratios N and Dua, Vivek and Bozinis, Nikolaos A and Bemporad, Alberto and Morari, Manfred},
  journal={Computers \& Chemical Engineering},
  volume={26},
  number={2},
  pages={175--185},
  year={2002},
  publisher={Elsevier}
}

@article{bemporad2002explicit,
  title={The explicit linear quadratic regulator for constrained systems},
  author={Bemporad, Alberto and Morari, Manfred and Dua, Vivek and Pistikopoulos, Efstratios N},
  journal={Automatica},
  volume={38},
  number={1},
  pages={3--20},
  year={2002},
  publisher={Elsevier}
}

@article{sakizlis2005explicit,
  title={Explicit solutions to optimal control problems for constrained continuous-time linear systems},
  author={Sakizlis, V and Perkins, JD and Pistikopoulos, EN},
  journal={IEE Proceedings-Control Theory and Applications},
  volume={152},
  number={4},
  pages={443--452},
  year={2005},
  publisher={IET}
}

@book{Pistikopoulos2007,
  title     = {Multi-Parametric Model-Based Control},
  editor    = {Pistikopoulos, Efstratios N. and Georgiadis, Michael C. and Dua, Vivek},
  publisher = {Wiley-VCH},
  year      = {2007},
  address   = {Weinheim, Germany}
}

@book{pistikopoulos2020multi,
  title={Multi-parametric optimization and control},
  author={Pistikopoulos, Efstratios N and Diangelakis, Nikolaos A and Oberdieck, Richard},
  year={2020},
  publisher={John Wiley \& Sons}
}

@incollection{kopp1962pontryagin,
  title={Pontryagin maximum principle},
  author={Kopp, Richard E},
  booktitle={Mathematics in Science and Engineering},
  volume={5},
  pages={255--279},
  year={1962},
  publisher={Elsevier}
}

@article{augustin2001computational,
  title={Computational sensitivity analysis for state constrained optimal control problems},
  author={Augustin, Dirk and Maurer, Helmut},
  journal={Annals of Operations Research},
  volume={101},
  number={1},
  pages={75--99},
  year={2001},
  publisher={Springer}
}

@book{stengel1994optimal,
  title={Optimal control and estimation},
  author={Stengel, Robert F},
  year={1994},
  publisher={Courier Corporation}
}

@book{gelfand2000calculus,
  title={Calculus of variations},
  author={Gelfand, Izrail Moiseevitch and Silverman, Richard A and others},
  year={2000},
  publisher={Courier Corporation}
}

@article{borrelli2005explicit,
  title={Explicit MPC for discrete hybrid systems},
  author={Borrelli, Francesco and Baoti{\'c}, Mato and Bemporad, Alberto and Morari, Manfred},
  journal={IEEE Transactions on Automatic Control},
  volume={50},
  number={12},
  pages={1919--1925},
  year={2005},
  publisher={IEEE}
}

@inproceedings{bemporad2003min,
  title={Min-max control of constrained uncertain discrete-time linear systems},
  author={Bemporad, Alberto and Borrelli, Francesco and Morari, Manfred},
  booktitle={European Control Conference (ECC)},
  pages={1651--1656},
  year={2003}
}

@article{sakizlis2004design,
  title={Design of robust model-based controllers via parametric programming},
  author={Sakizlis, Vassilis and Kakalis, Nikolaos MP and Dua, Vivek and Perkins, Julian D and Pistikopoulos, Efstratios N},
  journal={Automatica},
  volume={40},
  number={2},
  pages={189--201},
  year={2004},
  publisher={Elsevier}
}

@book{pontryagin1962mathematical,
  title={The Mathematical Theory of Optimal Processes},
  author={Pontryagin, Lev Semenovich and Boltyanskii, Vladimir Grigorevich and Gamkrelidze, Revaz Valerianovich and Mishchenko, Evgenii Frolovich},
  year={1962},
  publisher={Wiley},
  address={New York}
}

@incollection{kenefake2022ppopt,
  title={Ppopt-multiparametric solver for explicit mpc},
  author={Kenefake, Dustin and Pistikopoulos, Efstratios N},
  booktitle={Computer Aided Chemical Engineering},
  volume={51},
  pages={1273--1278},
  year={2022},
  publisher={Elsevier}
}

@article{pistikopoulos2015paroc,
  title={PAROC—An integrated framework and software platform for the optimisation and advanced model-based control of process systems},
  author={Pistikopoulos, Efstratios N and Diangelakis, Nikolaos A and Oberdieck, Richard and Papathanasiou, Maria M and Nascu, Ioana and Sun, Muxin},
  journal={Chemical Engineering Science},
  volume={136},
  pages={115--138},
  year={2015},
  publisher={Elsevier}
}

@book{luenberger1997optimization,
  title={Optimization by Vector Space Methods},
  author={Luenberger, David G.},
  year={1997},
  publisher={John Wiley \& Sons},
  address={New York}
}

@article{dombrovskii2018model,
  title={Model predictive control of constrained Markovian jump nonlinear stochastic systems and portfolio optimization under market frictions},
  author={Dombrovskii, Vladimir and Obyedko, Tatiana and Samorodova, Maria},
  journal={Automatica},
  volume={87},
  pages={61--68},
  year={2018},
  publisher={Elsevier}
}

@article{falcone2007predictive,
  title={Predictive active steering control for autonomous vehicle systems},
  author={Falcone, Paolo and Borrelli, Francesco and Asgari, Jahan and Tseng, Hongtei Eric and Hrovat, Davor},
  journal={IEEE Transactions on control systems technology},
  volume={15},
  number={3},
  pages={566--580},
  year={2007},
  publisher={IEEE}
}

@incollection{diehl2006fast,
  title={Fast direct multiple shooting algorithms for optimal robot control},
  author={Diehl, Moritz and Bock, Hans Georg and Diedam, Holger and Wieber, P-B},
  booktitle={Fast motions in biomechanics and robotics: optimization and feedback control},
  pages={65--93},
  year={2006},
  publisher={Springer}
}

@book{brenan1995numerical,
  title={Numerical solution of initial-value problems in differential-algebraic equations},
  author={Brenan, Kathryn Eleda and Campbell, Stephen L and Petzold, Linda Ruth},
  year={1995},
  publisher={SIAM}
}

@book{biegler2012control,
  title={Control and optimization with differential-algebraic constraints},
  author={Biegler, Lorenz T and Campbell, Stephen L and Mehrmann, Volker},
  year={2012},
  publisher={SIAM}
}

@article{srinivasan2003dynamic,
  title={Dynamic optimization of batch processes: I. Characterization of the nominal solution},
  author={Srinivasan, Balasubramaniam and Palanki, Srinivas and Bonvin, Dominique},
  journal={Computers \& Chemical Engineering},
  volume={27},
  number={1},
  pages={1--26},
  year={2003},
  publisher={Elsevier}
}

@INPROCEEDINGS{10886652,
  author={Abhijeet and Mohamed, Mohamed Naveed Gul and Sharma, Aayushman and Chakravorty, Suman},
  booktitle={2024 IEEE 63rd Conference on Decision and Control (CDC)}, 
  title={Convexity in Optimal Control Problems}, 
  year={2024},
  volume={},
  number={},
  pages={261-266},
  doi={10.1109/CDC56724.2024.10886652}}

@article{oberdieck2015explicit,
  title={Explicit hybrid model-predictive control: The exact solution},
  author={Oberdieck, Richard and Pistikopoulos, Efstratios N},
  journal={Automatica},
  volume={58},
  pages={152--159},
  year={2015},
  publisher={Elsevier}
}

@article{kouramas2013algorithm,
  title={An algorithm for robust explicit/multi-parametric model predictive control},
  author={Kouramas, Konstantinos I and Panos, Christos and Fa{\'\i}sca, Nuno P and Pistikopoulos, Efstratios N},
  journal={Automatica},
  volume={49},
  number={2},
  pages={381--389},
  year={2013},
  publisher={Elsevier}
}

@article{zavala2009advanced,
  title={The advanced-step NMPC controller: Optimality, stability and robustness},
  author={Zavala, Victor M and Biegler, Lorenz T},
  journal={Automatica},
  volume={45},
  number={1},
  pages={86--93},
  year={2009},
  publisher={Elsevier}
}

@article{zavala2008fast,
  title={Fast implementations and rigorous models: Can both be accommodated in NMPC?},
  author={Zavala, Victor M and Laird, Carl D and Biegler, Lorenz T},
  journal={International Journal of Robust and Nonlinear Control: IFAC-Affiliated Journal},
  volume={18},
  number={8},
  pages={800--815},
  year={2008},
  publisher={Wiley Online Library}
}

@article{de1995extension,
  title={An extension of Newton-type algorithms for nonlinear process control},
  author={De Oliveira, Nuno MC and Biegler, Lorenz T},
  journal={Automatica},
  volume={31},
  number={2},
  pages={281--286},
  year={1995},
  publisher={Elsevier}
}

@article{charitopoulos2017nonlinear,
  title={Nonlinear model-based process operation under uncertainty using exact parametric programming},
  author={Charitopoulos, Vassilis M and Papageorgiou, Lazaros G and Dua, Vivek},
  journal={Engineering},
  volume={3},
  number={2},
  pages={202--213},
  year={2017},
  publisher={Elsevier}
}

@incollection{charitopoulos2017closed,
  title={Closed loop integration of planning, scheduling and control via exact multi-parametric nonlinear programming},
  author={Charitopoulos, Vassilis M and Dua, Vivek and Papageorgiou, Lazaros G},
  booktitle={Computer aided chemical engineering},
  volume={40},
  pages={1273--1278},
  year={2017},
  publisher={Elsevier}
}

@article{srinivasan2003,
  title={Dynamic optimization of batch processes: II. Role of measurements in handling uncertainty},
  author={Srinivasan, Balasubramaniam and Bonvin, Dominique and Visser, Erik and Palanki, Srinivas},
  journal={Computers \& chemical engineering},
  volume={27},
  number={1},
  pages={27--44},
  year={2003},
  publisher={Elsevier}
}

@article{franccois2005use,
  title={Use of measurements for enforcing the necessary conditions of optimality in the presence of constraints and uncertainty},
  author={Fran{\c{c}}ois, Gregory and Srinivasan, Bala and Bonvin, Dominique},
  journal={Journal of Process Control},
  volume={15},
  number={6},
  pages={701--712},
  year={2005},
  publisher={Elsevier}
}


\end{document}